\documentclass{article}
\pdfoutput=1

\usepackage{amsmath}
\usepackage{amssymb}
\usepackage{anyfontsize}
\usepackage{amsthm}
\usepackage{amsfonts}
\usepackage[hidelinks]{hyperref}
\usepackage[english]{babel}
\usepackage{bbm}
\usepackage{bigints}
\usepackage{cite}
\usepackage{color}
\usepackage{float}
\usepackage{graphicx}
\usepackage[utf8]{inputenc}
\usepackage{mathptmx}
\usepackage{mathtools}
\usepackage{pgfplots} 
\usetikzlibrary{external}
\usepackage{subcaption}
\usepackage{stmaryrd}
\SetSymbolFont{stmry}{bold}{U}{stmry}{m}{n}
\usepackage{textcomp}
\usepackage{tikz}
\usepackage{url}

\newtheorem{Theorem}{Theorem}
\newtheorem{proposition}{Proposition}
\newtheorem{corollary}{Corollary}
\newtheorem{conjecture}{Conjecture}
\newtheorem{lemma}{Lemma}
\theoremstyle{definition}
\newtheorem{definition}{Definition}
\newtheorem{Remark}{Remark}
\newcommand{\enstq}[2]{\left\{#1\mathrel{}\middle|\mathrel{}#2\right\}}

\newcommand{\norm}[1]{\left\|#1\right\|}
\newcommand{\sinc}[0]{\textup{sinc}}

\newcommand{\N}{\mathbb{N}}
\newcommand{\Z}{\mathbb{Z}}

\newcommand{\R}{\mathbb{R}}

\newcommand{\Crad}{C^\infty_{c,rad}(B)}
\newcommand{\Lrad}{L^2_{rad}(B)}

\newcommand{\Hrad}{H^1_{rad}(B)}
\newcommand{\Hzrad}{H^1_{0,rad}(B)}
\newcommand{\rmin}{\delta_{\min}}
\newcommand{\rmax}{\delta_{\max}}

\newcommand{\abs}[1]{\left\lvert #1 \right\rvert}

\newcommand{\bs}[1]{\boldsymbol{#1}}
\newcommand{\varInRange}[4]{(#1_{#2})_{#3 \leq #2 \leq #4}}
\newcommand{\from}{\colon}
\newcommand{\Cinf}{C^{\infty}}
\newcommand{\isdef}{\mathrel{\mathop:}=}

\pgfplotsset{compat=newest}

\begin{document}

\author{Martin Averseng}

\title{ Fast discrete convolution in $\R^2$ using Sparse Bessel Decomposition }

\author{Martin Averseng}
\date{	\today  }

\definecolor{mycolor1}{rgb}{0.00000,0.44700,0.74100}%
\definecolor{mycolor2}{rgb}{0.85000,0.32500,0.09800}%
\definecolor{mycolor3}{rgb}{0.92900,0.69400,0.12500}%
\definecolor{mycolor4}{rgb}{0.49400,0.18400,0.55600}%

\maketitle
\noindent\abstract{\noindent We describe an efficient algorithm for computing the matrix vector products that appear in the numerical resolution of boundary integral equations in 2 space dimension. This work is an extension of the so-called Sparse Cardinal Sine Decomposition algorithm by Alouges et al., which is restricted to three-dimensional setups. Although the approach is similar, significant differences appear throughout the analysis of the method. Bessel decomposition, in particular, yield longer series for the same accuracy. We propose a careful study of the method that leads to a precise estimation of the complexity in terms of the number of points and chosen accuracy. We also provide numerical tests to demonstrate the efficiency of this approach. We give the compression performance for a $N\times N$ linear system for several values $N$ up to $10^7$ and report the computation time for the off-line and on-line parts of our algorithm. We also include a toy application to sound canceling to further illustrate the efficiency of our method.}

%
%

\section*{Introduction}
The Boundary Element Method requires the resolution of fully populated linear systems $Au = b$. As the number of unknowns gets large, the storage of the matrix $A$ and the computational cost for solving the system through direct methods (e.g. LU factorization) become prohibitive. Instead, iterative methods can be used, which require very fast evaluations of matrix vector products. In the context of boundary integral formulations, this takes the form of discrete convolutions:
\begin{equation}
	q_k = \sum_{l=1}^{N_z} G(\boldsymbol{z}_k - \boldsymbol{z}_l) f_l, \quad k \in \left\{1, \cdots, N_z\right\}.
	\label{discreteConv}					
\end{equation}
Here, $G$ is the Green's kernel of the partial differential equation under consideration, $\bs{z} = \varInRange{\bs{z}}{k}{1}{N_z}$  is a set of points in $\mathbb{R}^2$, with diameter $\rmax$, $\abs{\cdot}$ is the Euclidian norm and $f = \varInRange{f}{k}{1}{N_z}$ is a vector (typically the values of a function at the points $\bs{z}_k$). For example, the resolution of the Laplace equation with Dirichlet boundary conditions leads to \eqref{discreteConv} with $G(\bs{x}) = -\frac{1}{2\pi}\log \abs{\bs{x}}$ (the kernel of the single layer potential). 

In principle, the effective computation of the $(q_k)_{1 \leq k\leq N_z}$ using \eqref{discreteConv} requires $O(N_z^2)$ operations. However, several more efficient algorithms have emerged to compute an approximation of \eqref{discreteConv} with only quasilinear complexity in $N_z$. Among those are the celebrated Fast Multipole Method (see for example \cite{greengard1988rapid,rokhlin1990rapid, rokhlin1993diagonal, coifman1993fast, cheng1999fast} and references therein), the Hierarchical Matrix \cite{borm2003introduction}, and more recently, the Sparse Cardinal Sine Decomposition (SCSD) \cite{Alouges2015}.

One of the key ingredients in all those methods consists in writing the following local variable separation:
\[G(\bs{z}_k- \bs{z}_l) \approx \sum_j \lambda_j G^j_1(\bs{z}_k)G^j_2(\bs{z}_l),\] 
which needs to be valid for $\bs{z}_k$ and $\bs{z}_l$ arbitrarily distant from each other, and up to a controlled accuracy. This eventually results in compressed matrix representations and accelerated matrix-vector product. Notice that, to be fully effective, the former separation is usually made locally with the help of a geometrical splitting of the cloud of points $\bs{z}$ using a hierarchical octree.

Here we present an alternative compression and acceleration technique, which we call the Sparse Bessel Decomposition (SBD). This is an extension of the SCSD adapted to 2-dimensional problems. The SBD and SCSD methods achieve performances comparable to the aforementioned algorithms, they are flexible with respect to the kernel $G$, and do not rely on the construction of octrees, which makes them easier to implement. In addition, they express in an elegant way the intuition according to which a discrete convolution is nothing but a product of Fourier spectra. 

The method heavily relies on the Non Uniform Fast Fourier Transform (NUFFT) of type III (see the seminal paper \cite{NuFFT} and also \cite{greengard2004accelerating,poplau2006calculation} and references therein for numerical aspects and open source codes).  The NUFFT is a fast algorithm, which we denote by $\operatorname{NUFFT}_\pm[\bs{z},\boldsymbol{\xi}](\alpha)$ for later use, that returns, for arbitrary sets of $N_z$ points $\bs{z}$ and $N_\xi$ points $\bs{\xi}$ in $\mathbb{R}^2$ and a complex vector $\alpha  \in \mathbb{C}^{N_z}$, the vector $q \in \mathbb{C}^{N_\xi}$ defined by:
\[ q_\nu = \sum_{k = 1}^{N_z} e^{\pm i \bs{z}_k \cdot \boldsymbol{\xi}_\nu} \alpha_k, \quad \nu \in \left\{1,\cdots,N_\xi \right\}.\]
This algorithm generalizes the classical Fast Fourier Transform (FFT) \cite{cooley1965algorithm}, to nonequispaced data, preserving the quasi-linear complexity in $N_{z,\xi} \isdef \max(N_z,N_\xi)$.
The SBD method first produces a discrete and sparse approximation of the spectrum of $G$,
\begin{equation}
	\label{Gapprox}
	G(\bs{x}) \approx G_{\text{approx}}(\boldsymbol{x}) \isdef \sum_{\nu=1}^{N_\xi} e^{i  \boldsymbol{x}\cdot \boldsymbol{\xi}_\nu} \hat{\omega}_\nu, \quad \abs{\bs{x}} \leq \rmax.
\end{equation}
This approximation is replaced in \eqref{discreteConv} to yield 
\begin{equation}
	\begin{split}	q_{k} &\approx \left(\sum_{\nu = 1}^{N_\xi} e^{+i  \bs{z}_k  \cdot\boldsymbol{\xi}_\nu } \left[\hat{\omega}_\nu \sum_{l=1}^{N_z} e^{- i \bs{z}_l \cdot \boldsymbol{\xi}_\nu} f_l) \right]\right)_{1 \leq k \leq N_z}\\
		&= \operatorname{NUFFT}_+[\bs{z},\boldsymbol{\xi}]\left(\hat{\omega} \odot \operatorname{NUFFT}_-[\bs{z},\boldsymbol{\xi}]\big(f\big)\right).
	\end{split}
	\label{far convolution}					
\end{equation}
where $\odot$ denotes the elementwise product between vectors. The decomposition \eqref{Gapprox} is obtained "off-line" and depends on the value of $\rmax$, but is independent of the choice of the vector $f$ and can thus be used for any evaluation of the matrix vector product \eqref{discreteConv}. 
The approximation \eqref{far convolution} reduces the complexity from $O(N_z^2)$ to $O(N_{z,\xi}\log (N_{z,\xi}))$, stemming from the NUFFT complexity.

The NUFFT has already been used in the literature for the fast evaluation of quantities of the form \eqref{discreteConv}. In particular, our algorithm shares many ideas with the approach in \cite{potts2004fast}. The method presented therein also relies on an approximation of the form \eqref{Gapprox}. However, we choose the set of frequencies $\bs{\xi}$ in a different way, leading to a sparser representation of $G$ (see \autoref{RemarkPotts}). 

The kernel $G$ is usually singular near the origin, in which case \eqref{Gapprox} can only be accurate for $|\bs{z}|$ above some threshold $\rmin$. Part of the SBD algorithm is thus dedicated to computing a local correction using a sparse matrix (which we call the close correction matrix, denoted by $D$ in the sequel) to account for the closer interactions. This threshold must be chosen so as to balance the time spent for computing the far \eqref{far convolution} and those close contributions. As a matter of fact, we shall prove the following:

\begin{Theorem} Assume the points $\bs{z}$ are uniformly distributed on a regular curve, and $G(\bs{x}) = \log \abs{\bs{x}}$. Let $\varepsilon > 0$ the desired accuracy of the method, and assume $N_z > \abs{\log{\varepsilon}}$. Fix 
	\[a =\dfrac{\abs{\log\varepsilon}^{2/3}}{N_z^{2/3 - \alpha}}\]
	for some $\alpha \in \left[0,\frac{1}{6}\right]$, and choose 
	\[\rmin = a \rmax.\] 
	Then there exists a constant $C>0$ independent of $N_z$, $\varepsilon$ and $\alpha$ such that:
	\label{The:GlobalComplexity}
	\begin{itemize}
		\item[(i)] The number of operations required for the computation of the representation \eqref{Gapprox} valid for $\abs{\bs{x}} \in [\rmin,\rmax]$  is bounded by 
		      \[ C_{\textup{off}}(N_z,\varepsilon,\alpha) \leq C \abs{\log \varepsilon} N_z^{2 - 3\alpha};\]
		\item[(ii)] The number of operations required for the assembling of the close correction matrix $D$ is bounded by
		      \[C_{\textup{assemble}}(N_z,\varepsilon,\alpha)\leq C \abs{\log\varepsilon}^{2/3}  C_{\textup{NUFFT}}(\varepsilon) N_z^{4/3 + \alpha}\log(N_z);\]
		\item[(iii)] Once these two steps have been completed, \eqref{discreteConv} can be evaluated for any choice of vector $f$ at a precision at least $\varepsilon \displaystyle\sum_l \abs{f_l}$ in a number of operations bounded by
		      \[C_{\textup{on}}(N_z,\varepsilon,\alpha) \leq C \abs{\log\varepsilon}^{2/3} \left(  N_z ^{4/3 + \alpha} + C_{\textup{NUFFT}}(\varepsilon)N_z^{4/3 - 2\alpha}\log N_z \right).\] 
	\end{itemize} 
\end{Theorem}

We prove \autoref{The:GlobalComplexity} using several steps. We first introduce the Fourier-Bessel series (\autoref{sec:FourierBesselSeries}). In \autoref{sec:SBD}, we present the Sparse Bessel Decomposition, and analyze its numerical stability. In section \autoref{sec:ApplicationLaplaceHelmholtz}, we apply the SBD method to the Laplace kernel, and give estimates for the number of terms to reach a fixed accuracy. We also explain how to adapt this method for other kernels. In \autoref{sec:circular}, we show how to convert a SBD decomposition into an approximation of the form \eqref{Gapprox} through what we call "circular quadratures" and provide error estimates. In \autoref{sec:complexities}, we summarize the complexities of each step and prove \autoref{The:GlobalComplexity}. We conclude with some numerical examples. 

Before this, let us start by a brief overview of our algorithm.

\section{Summary of the algorithm}
\setcounter{equation}{0}
\numberwithin{equation}{section} 
\label{sec:overview}

The SBD algorithm can be summarized as follows:
\paragraph{Off-line part}
\begin{itemize}
	\item[]\textbf{Inputs:} A radially symmetric kernel $G$, a set of $N_z$ points $\bs{z}$ in $\mathbb{R}^2$ of diameter $\rmax$, a value for the parameter $\rmin$, a tolerance $\varepsilon > 0$.
	\item[]\textbf{Sparse spectrum sampling:} Compute a set of $N_\xi$ complex weights $\hat{\omega}$ and $N_\xi$ frequencies $\boldsymbol{\xi}$ so that \eqref{Gapprox} is valid for $\rmin \leq |\boldsymbol{x}| \leq \rmax$ up to tolerance $\varepsilon$. 
	\item[]\textbf{Correction Matrix:} Determine the set $\mathcal{P}$ of all pairs $(k,l)$ such that $\abs{\bs{z}_k - \bs{z}_l} \leq \delta_{\min}$ (fixed-radius neighbor search). Assemble the close correction sparse matrix:
	      \begin{equation}
	      	\label{defD}
	      	D_{kl} = \delta_{(k,l) \in \mathcal{P}} \left( G({\bs{z}_k - \bs{z}_l}) - \sum_{\nu = 1}^{N_{\xi}}e^{i (\bs{z}_k - \bs{z}_l)\cdot \boldsymbol{\xi}_\nu} \hat{\omega}_\nu\right).
	      \end{equation}
	      Notice that the second term is a non-uniform discrete Fourier transform. Indeed, if we introduce $(\bs{y}_p)_{p \in \mathcal{P}}$ given by 
	      $\bs{y}_{(k,l)} = \bs{z}_k - \bs{z}_l$, and $\bs{Y} = \textup{NUFFT}_+[\bs{y},\bs{\xi}](\hat{\omega})$, the non-zero entries of $D$ are given by
	      \[ D_{k,l} = G(\bs{y}_{(k,l)}) - \bs{Y}_{(k,l)}.\]
	\item[] \textbf{Outputs}: The set of weights $\hat{\omega}$, the frequencies $\boldsymbol{\xi}$ and the sparse matrix $D$. 
\end{itemize}
\paragraph{On-line part}
\begin{itemize}
	\item[] \textbf{Input:} All outputs of the off-line part, and a complex vector $f$ of size $N_z$. 
	\item[] \textbf{Far approximation:} Compute, for all $k$,
	      \begin{equation}
	      	\label{FarApprox}
	      	q^{\text{far}} = \sum_{l=1}^{N_z} G_{\textup{approx}}(\bs{z}_k - \bs{z}_l) f_l.
	      \end{equation} 
	      For this, follow three steps
	      \begin{itemize}
	      	\item[(i)] \textbf{Space $\rightarrow$ Fourier: } Compute $\hat{f} = \textup{NUFFT}_-[\bs{z},\boldsymbol{\xi}](f).$
	      	\item[(ii)] \textbf{Fourier multiply:} Perform elementwise multiplication by $\hat{\omega}$:
	      	\[{\hat{g}_{\nu} = \hat{\omega}_\nu \hat{f_\nu}.}\]
	      	\item[(iii)] \textbf{Fourier $\rightarrow$ Space: } Compute $q^{\text{far}} =  \textup{NUFFT}_+[\bs{z},\boldsymbol{\xi}](\hat{g}).$
	      \end{itemize}
	\item[] \textbf{Close correction:} Compute the sparse matrix product:
	      \[q^{\textup{close}} = Df.\]
	\item[] \textbf{Output:} The vector $q = q^{\textup{far}} + q^{\textup{close}}$, with, for any $k \in \left\{1,\cdots,N_z\right\},$	
	      \[ \abs{q_k - \sum_{l = 1}^{N_z} G({\bs{z}_k - \bs{z}_l}) f_l} \leq \varepsilon \sum_{l=1}^{N_z} \abs{f_l}.\]
\end{itemize}

\subsection*{\textbf{The sparse spectrum sampling step:}}
The main novelty in our algorithm is the method for producing an approximation of the form \eqref{Gapprox}. We proceed in two steps, uncoupling the radial and circular coordinates. 
\paragraph{Sparse Dessel Decomposition:} In a first part, we compute a Bessel series approximating $G(r)$ on $[\rmin,\rmax]$, which we call a Sparse Bessel Decomposition. It takes the form
\[G(r) \approx G(\rmax) + \sum_{p=1}^P \alpha_p J_0(\rmax \rho_p r), \quad r \in [\rmin,\rmax],\]
where $J_0$ is the Bessel function of first kind and order zero and $(\rho_p)_{p \in \N*}$ is the sequence of its roots (see \autoref{defJ0} for more details). In order to compute the coefficients $\varInRange{\alpha}{p}{1}{P}$, we first choose a starting value for $P$ and compute the weights $\alpha_1,\cdots, \alpha_{P}$ that minimize the least square error
\[\bigintsss_{\rmin \leq |\bs{x}|\leq \rmax} \abs{\nabla \left(G(\bs{x}) - \sum_{p=1}^P \alpha_p J_0(\rho_p \rmax \abs{\bs{x}})\right)}^2 d\bs{x}.\]
This amounts to solving an explicit linear system. We keep increasing $P$ until the residual error goes below the required tolerance. To choose the stopping criterion, we suggest monitoring the $L^{\infty}$ error near $\rmin$ where it is usually the highest. The successive choices of $P$ are made using a dichotomy search. 
\paragraph{Circular quadrature:}In a second step, we use approximations of the form:
\[J_0(\rho_p \abs{\bs{x}}) \approx \dfrac{1}{M_p}\sum_{m=0}^{M_p-1}e^{i \rho_p \bs{\xi}_p^m \cdot \bs{x}}, \quad p \in \{1,\cdots,P\},\]
which are discrete versions of the identity:
\[ J_0(\rho_p\abs{\bs{x}}) = \frac{1}{2\pi}\int_{\abs{\bs{\xi}}=1}{e^{i \rho_p\bs{x} \cdot \bs{\xi}}} d\sigma(\bs{\xi}).\]
We sum them to eventually obtain the formula \eqref{Gapprox}. 
\begin{Remark}
	In the SCSD method \cite{Alouges2015}, the Bessel functions are replaced by cardinal sine functions, since, for $\bs{x} \in \R^3$
	\[ \frac{1}{4\pi}\int_{\abs{\bs{\xi}}=1} e^{i \bs{x}\cdot \bs{\xi}} d\sigma(\bs{\xi}) = \sinc(\abs{\bs{x}}),\]
	where the integral is now taken over $\mathbbm{S}^2 \subset \R^3$.
\end{Remark}

\section{Series of Bessel functions and error estimates}
\label{sec:FourierBesselSeries}
In this section, we give a short introduction to Fourier-Bessel series. A possible reference on this topic is \cite{watson1995treatise}, chapter XVIII. 
The main result needed for our purpose is \autoref{DecroissanceFourierBessel},  an equivalent statement of which can be found in to Theorem $1$ in \cite{tolstov2012fourier} chapter 8, section 20. 

In \autoref{FunctionalFramework}, we quickly recall some classical facts on the Laplacian eigenfunctions with Dirichlet boundary conditions on the unit ball and formulate a conjecture (\autoref{ConjCp}) for the normalization constant supported by strong numerical evidence (\autoref{figure:encadrementCp}). It is worth noting that the use Laplace eigenfunctions as the decomposition basis and the results we obtain hereafter do not rely on the space dimension. For example, in $\R^3$ the radial eigenvalues of the Laplacian are proportional to
\[  \bs{x} \mapsto \dfrac{J_{1/2}(2\pi p\abs{\bs{x}})}{\bs{|x|}^{1/2}}, \quad  p\in\N^*.\] 
Therefore, our approach generalizes \cite{Alouges2015} to any dimension.  
\subsection{Radial eigenvalues of the Laplace operator with Dirichlet conditions}
\label{FunctionalFramework}
In the following, $B$ denotes the unit ball in $\mathbb{R}^2$, and $\mathcal{C}$ its boundary. We say that a function $u\from\mathbb{R}^2\rightarrow \mathbb{R}$ is radial if there exists $\tilde{u}\from\mathbb{R}^+\to \R$ such that for any $\bs{x} \in \R^2$, 
\[ u(\bs{x}) = \tilde{u}(\abs{\bs{x}}).\] 
In this case, we use the notation $u(\bs{x})= u(r)$ where $\abs{\bs{x}} = r$. 
We note $\Lrad$ the closed subspace of $L^2$ that are radial functions and $\Crad = \enstq{\varphi \in \Cinf_c(B)}{\varphi \text{ is radial }}$. Similarly,
\[\Hrad = \enstq{ u \in \Lrad}{ \forall 1 \leq i \leq 2, \dfrac{\partial u}{\partial x_i} \in L^2(B)},\]
which is a Hilbert space for the norm
\[\norm{u}_{\Hrad}^2 = \int_{B}\abs{u}^2 + \abs{\nabla u}^2 = 2\pi\int_{0}^1 r \left(u(r)^2 + u'(r)^2\right) dr. \]
Finally, we note $\Hzrad$ the closure of $\Crad$ in $\Hrad$, with the norm
\[\norm{u}_{\Hzrad}^2 = {2\pi \int_{0}^1 r u'(r)^2dr}.\] 

We now briefly recall some facts on Bessel functions. All the results that we use on this topic can be found in the comprehensive book \cite{abramowitz1964handbook}, most of which can be consulted handily on the digital library \cite{NIST:DLMF}.

\begin{definition}
	\label{defJ0}
	The Bessel function of the first kind and order $\alpha$, $J_\alpha$ is defined by the following series: 
	\begin{equation}
		\label{J0powerSeries}
		J_\alpha(r) \isdef \sum_{m=0}^\infty \frac{(-1)^m}{m! \, \Gamma(m+1+\alpha)} {\bigg(\frac{r}{2}\bigg)}^{2m+\alpha}\!\!\!\!\!\!\!\!\!\!\!\!.
	\end{equation}
\end{definition}
\noindent $J_0$ is a $\Cinf$ solution of Bessel's differential equation 
\begin{equation}
	\label{BesselDifferentialEquation}
	r^2f''(r) + r f'(r) + r^2 f(r) = 0.
\end{equation}
The roots $(\rho_p)_{p \in \N^*}$ of $J_0$, behave, for large $p$, as 
\[ \rho_p \underset{p \to \infty}{\sim} \pi p.\]
More precisely, $\rho_p = \pi p - \frac{\pi}{4} + O\left(\frac{1}{p}\right)$, and
\begin{equation}
	\pi(p - 1/4)\leq \rho_p \leq \pi(p - 1/8).
	\label{EncadrementRhop}
\end{equation}

\noindent For any $p\in \N^*$, we introduce:
\[e_p(\bs{x}) = C_p J_0(\rho_p \abs{\bs{x}}),\]
where the normalization constant $C_p$ is chosen such that $\norm{e_p}_{\Hrad} = 1$, that is  
\[C_p = \dfrac{1}{\left(2\pi \int_B  r \rho_p^2 J_1 (\rho_p r)^2\right)^{1/2}} = \dfrac{1}{\sqrt{\pi}\rho_p\abs{J_1(\rho_p)}}.\]
\noindent For any $p \in \N^*$, $e_p$ satisfies:
\begin{equation}
	\label{epEstUnVP}
	-\Delta e_p = \rho_p^2 e_p.
\end{equation}
The following result is well-known.
\begin{Theorem} 
	\label{epEstUneBaseDeHilbert}
	The family $\left\{e_p, p\in \N^*\right\}$ is a Hilbert basis of $\Hzrad$.
\end{Theorem}

One can check, using asymptotic expansions of Bessel functions, that
\begin{equation}
	\label{equivalentCp}
	C_p = \dfrac{1}{\sqrt{2 \pi p}} + O\left(\frac{1}{p^{3/2}}\right), 
\end{equation}
We will actually need a more precise knowledge on the constant $C_p$. We formulate the next conjecture, which seems hard to establish and for which we didn't find any element of proof in the literature. Numerical evidence exposed in \autoref{figure:encadrementCp} strongly suggests it is true. 
\begin{conjecture} 
	\label{ConjCp}
	For all $p \in \N^*$:
	\begin{equation}
		\label{EncadrementCp}
		\frac{1}{\sqrt{2\pi p}} \leq C_p \leq \frac{1}{\sqrt{2\pi (p-1/4)}}.
	\end{equation}
\end{conjecture}

\begin{figure}[H]
	\centering
	\includegraphics[scale = 1]{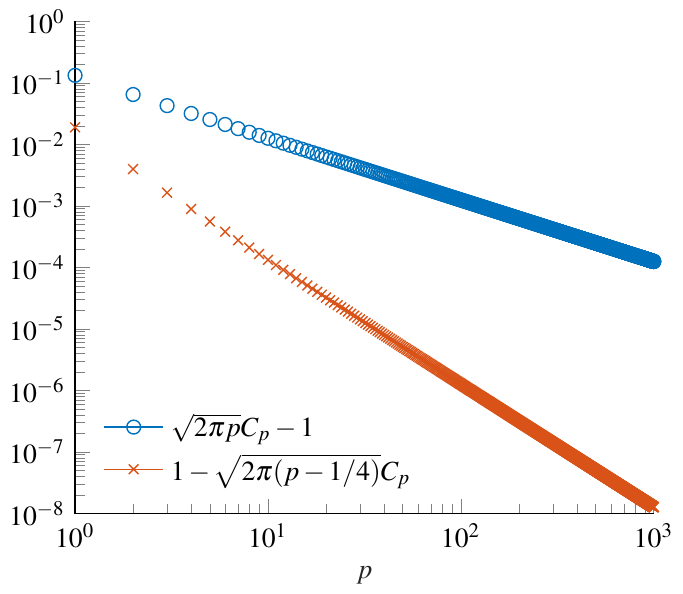}
	\caption{Illustration of \autoref{ConjCp} with numerical values of the first $1000$ terms of the sequence ${v_p = \sqrt{2\pi p}C_p - 1}$ (blue circles) and ${w_p = 1 - \sqrt{2\pi(p-1/4)}C_p}$ (red crosses) in log-log scale}
	\label{figure:encadrementCp}
\end{figure}

\subsection{Truncature error for Fourier-Bessel series of smooth functions}
\label{FourierBesselTruncError}
We now introduce the Fourier-Bessel series and prove a bound for the norm of the remainder. 
In \autoref{epEstUneBaseDeHilbert}, we have shown that any function $f \in \Hzrad$ can be expanded through its so-called Fourier-Bessel series as
\[f = \sum_{p\in \mathbb{N}^*}c_p(f)e_{p}.\]
The generalized Fourier coefficients are obtained by the orthonormal projection: 
\[c_p(f) = \displaystyle \int_B \nabla f(\bs{x}) \cdot \nabla e_{p}(\bs{x}) d\bs{x} = \rho_k^2 \int_{B}{f(\bs{x}) e_p(\bs{x})d\bs{x}}.\]
Most references on this topic focus on proving pointwise convergence of the series even for not very regular functions $f$ (e.g. piecewise continuous, square integrable, etc.) \cite{stempak2002convergence,guadalupe1993mean,Balodis1999,colzani1993equiconvergence}.  In such cases, the Fourier-Bessel series may exhibit a Gibb's phenomenon \cite{wilton1928gibbs}. On the contrary here, we need to establish that the Fourier-Bessel series of very smooth functions converges exponentially fast. To this aim, we first introduce the following terminology: 
\begin{definition}
	We say that a radial function $f$ satisfies the multi-Dirichlet condition of order $n \in \N^*$ if $f$ is $H^{2n}$ in a neighborhood of $\mathcal{C}$ and if for all $s \leq n-1$, the $s-th$ iterate of $-\Delta$ on $f$, denoted by $(-\Delta)^s f$, vanishes on $\mathcal{C}$ (with the convention $(-\Delta)^0 f = f$). 
\end{definition}
\begin{proposition} 
	\label{DecroissanceFourierBessel}
	If $f \in H^{2n}(B)$ satisfies the multi-Dirichlet condition of order $n$, then for any $p \in \mathbb{N}^*$:
	\[ c_p(f) = \dfrac{1}{\rho_{p}^{2n-2}} \int_{B}\left(-\Delta\right)^{n} f(\bs{x})e_p(\bs{x}) d\bs{x}.\] 
\end{proposition}
\begin{proof}
	If $f$ satisfies the multi-Dirichlet condition of order $n=1$, then by integration by parts:
	\[c_p(f) = \int_{B}(-\Delta)f(\bs{x}) e_p(\bs{x})d\bs{x},\]
	since $e_p$ vanishes on $\mathcal{C}$.
	Assume the result is true for some $n \geq 1$ and let $f$ satisfy the multi-Dirichlet condition of order $n+1$. Then, using the fact that $e_p$ is an eigenvector of $-\Delta$ associated to the eigenvalue $\rho_p^2$ we obtain:
	\[c_p(f) = \frac{1}{\rho_p^{2n}}\int_{B}(-\Delta)^{n}f(\bs{x})~ (-\Delta) e_p(\bs{x}) d\bs{x}.\]
	The result follows from integration by parts where we successively use that $(-\Delta)^{n}f$ and $e_p$ vanish on $\mathcal{C}$.
	
\end{proof}

\begin{corollary} If $f \in H^{2n}(B)$ satisfies the multi-Dirichlet condition of order $n$, there exists a constant $C$ independent of the function $f$ such that for all $p \in \mathbb{N^*}$, 
	\[ |c_p(f)| \leq  C \dfrac{\norm{(-\Delta)^n f}_{\Lrad}}{(\pi p)^{2n-1}}.\] 
\end{corollary}
\noindent Notice that this is similar to the fact that the Fourier coefficients of smooth functions decay fast. 
\begin{proof}
	We apply the result of the previous proposition and remark that since $e_p$ is an eigenfunction of the Laplace operator with unit norm in $\Hzrad$, $\norm{e_p}_{\Lrad} = \dfrac{1}{\rho_p}\norm{e_p}_{\Hzrad} = \dfrac{1}{\rho_p}$. To conclude, recall $\rho_{p} \sim p\pi$ for large $p$.
	
\end{proof}			
\begin{corollary} Let the remainder be defined as 
	\[R_P(f) = \displaystyle\sum_{p = P+1}^{+\infty} c_{p}(f) e_{p}.\]
	If $f \in H^{2n}(B)$ satisfies the multi-Dirichlet condition of order $n$, there exists a constant $C$ independent of $n$ and $P$ such that: 
	\[\norm{R_P(f)}_{\Hzrad} \leq C\dfrac{\norm{(-\Delta)^n f}_{\Lrad}}{(\pi P)^{2n}}\sqrt{\dfrac{P^3}{n}}.\]
	\label{EstimationRest}
\end{corollary}																												
\begin{proof}
	Parseval's identity implies
	\[\norm{R_P(f)}_{\Hzrad}^2 = \sum_{p=P+1}^{+\infty}|c_p(f)|^2.\]																																			
	According to the previous results, we find that:
	\[\norm{R_P(f)}_{\Hzrad} \leq C \norm{(-\Delta)^n f}_{\Lrad}\sqrt{\sum_{p= P+1}^{+\infty} \dfrac{1}{(\pi p)^{4n-2}}}.\]
	The announced result follows from $\displaystyle\sum_{p > P} \frac{1}{p^{\alpha}} \propto \frac{1}{(\alpha - 1)P^{\alpha-1}}$ for $\alpha > 1$. 
	
\end{proof}

\subsection{Other boundary conditions}
\label{Robin}
When we replace the Dirichlet boundary condition by the following Robin boundary conditions
\begin{equation}
	\label{robinCondition}
	\dfrac{\partial u}{\partial n} + H u = 0
\end{equation}
for some constant $H \geq 0$, the same analysis can be conducted, leading to Dini series (also covered in \cite{watson1995treatise}). This time, we construct a Hilbert basis of $\Hrad$ with respect to the bilinear form
\[a_H(u,v) \isdef \int_{B} \nabla u(\bs{x} ) \cdot \nabla v(\bs{x} ) d\bs{x} + H \int_{\mathcal{C}}{u(\bs{x} )v(\bs{x} )}d\sigma(\bs{x} ).\]
The following result holds. 
\begin{Theorem}
	Let $(\rho_p^H)_{p \in \N*}$ the sequence of positive solutions of
	\[r J_0'(r) + H J_0(r) = 0.\]
	\begin{itemize}
		\item[(i)] If $H>0$, the functions 
		\[e_p^H(r) = C_p J_0(\rho_p^H r),\]
		with $C_p$ such that $a_H(e_p^H,e_p^H) = 1$, form a Hilbert basis of $H^1_{\text{rad}}(B)$. 
		\item[(ii)]If $H = 0$, a constant function must be added to the previous family to form a complete set. 
	\end{itemize}
\end{Theorem}
\noindent It can be checked that the truncature error estimates in \autoref{EstimationRest} extend to this case, for functions satisfying multi-Robin conditions of order $n \geq 1$, that is for all $s\leq n-1$, $(-\Delta)^s u$ satisfies \eqref{robinCondition}.

\section{Sparse Bessel Decomposition}

\label{sec:SBD}
\subsection{Definition of the SBD}
										
Consider the kernel $G$ involved in \eqref{discreteConv}. We can assume up to rescaling $G$ that the diameter $\rmax$ of $\bs{z}$ is bounded by $1$, and therefore, we need to approximate $G$ only on the unit ball $B$. 
If we wish to approximate $G$ in series of Bessel functions, two kinds of complications are encountered:
\begin{itemize}
	\item[(i)] $G$ is usually singular near the origin, therefore not in $H^{2n}(B)$ (even for $n=1$). 
	\item[(ii)] The multi-Dirichlet conditions may not be fulfilled up to a sufficient order.
\end{itemize}

The point (ii) is crucial in order to apply the error estimates of the previous section. The first two kernels that we will study (Laplace and Helmholtz) satisfy the favorable property:
\[\Delta G = \lambda G\]
for some $\lambda \in \mathbb{C}$, which will be helpful to ensure (ii) at any order. For more general kernel, we propose in \autoref{begal1} a simple trick to enforce multi-Dirichlet conditions up to a given order. 

As for point (i), we will use the following method: for the approximation 
\[G \approx \sum_{p = 1}^P \alpha_p e_p,\]
we know by \autoref{epEstUneBaseDeHilbert} that the minimal $H^1_0$ error on $B$ is reached for $\alpha_p = c_p(G)$. However, if the closest interaction in \eqref{discreteConv} are computed explicitly, it can be sufficient to approximate $G$ in a domain of the form $a \leq r \leq 1$ for some $a$. For this reason, we propose to define the coefficients $(\alpha_1,\cdots,\alpha_P)$ as the minimizers of the quadratic form
\[ Q^P(t_1,t_2,...,t_P) = \bigintsss_{\mathcal{A}(a)} \left|\nabla \!\left( G(\bs{x}) - \sum_{p=1}^P t_p J_0(\rho_p |\bs{x}|)\right)\right|^2d\bs{x},\]
where $\mathcal{A}(a)$ is the annulus $\enstq{\bs{x} \in \R^2}{a < \abs{\bs{x}} < 1}$. In the sequel, those coefficients will be called the SBD coefficients of $G$ of order $P$. Obviously, for any radial function $\tilde{G}$ defined on $B$ that coincides with $G$ on $\mathcal{A}(a)$, one has 
\[ Q^P(\alpha_1,\cdots,\alpha_P) \leq \bigintsss_{B} \left|\nabla \!\left( \tilde{G}(\bs{x}) - \sum_{p=1}^P c_p(\tilde{G}) J_0(\rho_p |\bs{x}|)\right)\right|^2d\bs{x}. \]
In particular, when $\tilde{G}$ is smooth up to the origin, this gives an error estimate via \autoref{EstimationRest}. If we choose a sufficiently high value for $a$, we ensure that smooth enough extensions exist, and ensure fast decay of the coefficients. 			
\begin{Remark}
	\label{RemarkPotts}
	The SBD weights do not depend on any specific extension of $G$ outside the annulus. Therefore, they provide the sparsest approximation one can expect, contrary the the usual approach where an explicit regularization $\tilde{G}$ of the kernel is constructed and the coefficients $c_p(\tilde{G})$ are used (see, for example \cite{potts2004fast}). 
\end{Remark}									
The next result shows that the $H_0^1$ norm on $\mathcal{A}(a)$ controls the $L^{\infty}$ norm, thus ruling out any risk of Gibb's phenomenon.
\begin{lemma}
	Let $a\in (0,1)$ and $e\in H^1_{\text{rad}}(\mathcal{A}(a))$ that vanishes on $\mathcal{C}$. 
	Then $e$ coincides almost everywhere with a continuous function with
	\[\abs{e(\bs{x}_0)} \leq \sqrt{\dfrac{-\log \abs{\bs{x}_0}}{2\pi}}\sqrt{\int_{\mathcal{A}(a)} \abs{\nabla e(\bs{x})}^2} d\bs{x},\quad  \text{almost for all }\bs{x}_0 \in \mathcal{A}(a).\]
\end{lemma}
\begin{proof}
	It is sufficient to show the inequality for smooth $e$, the general result following by density. Let $\bs{x}_0 \in \mathcal{A}(a)$, we have, since $e(1)=0$:
	\begin{align}
		\abs{e(\bs{x}_0)} & \leq \int_{\abs{\bs{x}_0}}^1 \abs{e'(r)}dr,                                                                     \\
		                  & \leq \sqrt{2\pi \int_{\abs{\bs{x}_0}}^1 r \abs{e'(r)}^2 dr} \sqrt{\int_{\abs{\bs{x}_0}}^1 \frac{1}{2\pi r} dr}.
	\end{align}	
					
\end{proof}

\subsection{Numerical computation of the SBD}
\label{sub:Chol}
																		
For a given kernel $G$, the SBD coefficients $\alpha_p$ are obtained numerically by solving the following linear system: 
\begin{equation}
\begin{split}
\sum_{q = 1}^P \left(\int_{\mathcal{A}(a)} \rho_ q J_1(\rho_p|\bs{x}|) J_1(\rho_q|\bs{x}|) d\bs{x}\right) \alpha_q \\
\quad = -\int_{\mathcal{A}(a)} G'(\bs{x}) J_1(\rho_p|\bs{x}|)d\bs{x}, \quad 1\leq p \leq P,
\end{split}	
\label{LinearSystem}
\end{equation}
Where $J_1$ is the Bessel function of first kind and order $1$ (in fact, $J_0' = - J_1$). We solve this system for increasing values of $P$ until a required tolerance is reached. It turns out that the matrix $A^P$ whose entries are given by
\[ A^P_{k,l} = \int_{\mathcal{A}(a)} \!\!\!\!\!\!\!\! \nabla e_k \cdot \nabla e_l,\quad k,l \in \{1,\cdots,P\},\]
is explicit: for $(i,j) \in \{1,\cdots,P\}^2$, the non-diagonal entries of $A^P$ are
\begin{equation*}
	A_{i,j} = \frac{2\pi C_i C_j \rho_i \rho_j}{\rho_j^2 - \rho_i^2}\bigg[F_{i,j}(1) - F_{j,i}(1) - F_{i,j}(a) + F_{j,i}(a)\bigg],
\end{equation*}
where 
\[	 F_{i,j}(r) =  \rho_i r J_0(\rho_i r)J_0'(\rho_j r),\]
while the diagonal entries are
\begin{equation*}
	A_{i,i} = 2\pi C_i^2 \big(F_i(1) - F_i(a)\big),
\end{equation*}
where 
\[F_i(r) = \rho_i^2r^2\left[\dfrac{1}{2}J_0(\rho_ir)^2 + \frac{1}{2}J_0'(\rho_ir)^2\right] + \rho_irJ_0(\rho_i r)J_0'(\rho_ir).\]
Those formulas are obtained using Green's formulas together with the fact that $e_k$ are eigenfunctions of the Laplace operator. They are valid for any value of $\rho_k$ (not just the roots of $J_0$).
																		
\subsection{Conditioning of the linear system}
\newcommand{\Pa}{\gamma}
\newcommand{\Pastar}{\Pa^*}
The conditioning of $A$ seems to depend almost exclusively on the parameter $\Pa \isdef Pa$. We were only able to derive an accurate estimate of the conditioning of $A$ when $\Pa$ is small enough. For large $\Pa$, we will show some numerical evidence for a conjectured bound on the conditioning of $A$.

\subsubsection*{Conditioning of $A$ for small $\Pa$}
\begin{Theorem} If \autoref{ConjCp} is true, then, for $b=1$, the eigenvalues of $A$ lie in the interval~ ${[F(\Pa) - \frac{\pi^4}{144}\frac{\gamma^4}{P},1]}$ where
	\[F(\Pa) = 1 - \int_{0}^{\pi \gamma} \frac{t}{2}(J_1(t)^2 - J_0(t)J_2(t) )dt.\]
	\label{The:lowBoundCon}
\end{Theorem}
This estimate is only useful when ${F(\Pa) > 0}$, that is ${\Pa < \Pastar}$ where $\Pastar$ is the first positive root of $F(z)$. One has 
\[\Pastar \approx 1.471.\]	
In particular, for $\gamma = 1$, The matrix $A$ is well conditioned, the ratio of its largest to its smallest eigenvalues being of the order $F(1)^{-1} < 2$.
A plot of $F$ is provided below, \autoref{figure:Fconditionnement}, and some numerical approximations of the minimal eigenvalue of $A$ are shown in function of $\Pa$ for several values of $P$. 		
\begin{figure}[ht]
	\centering			
	\includegraphics[scale = 1]{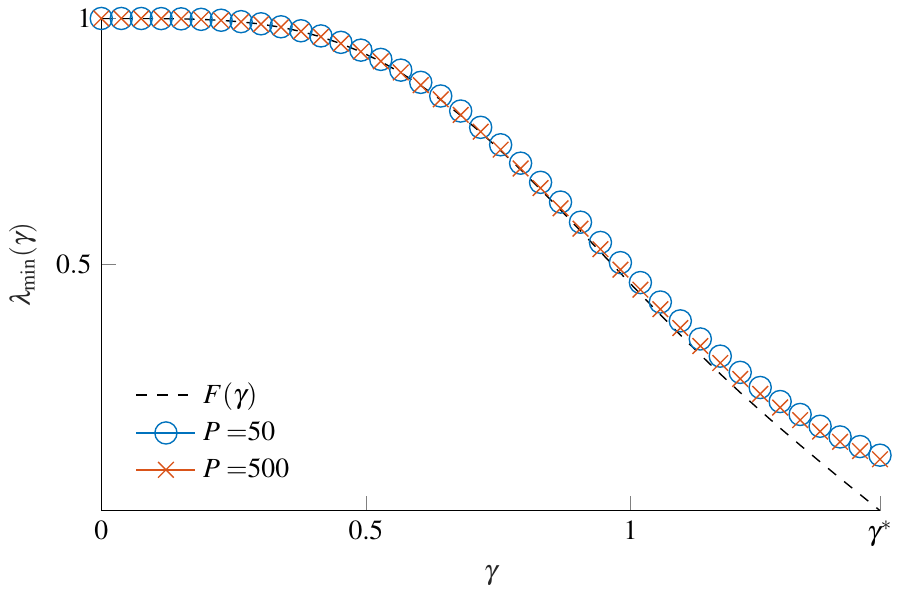}
	\caption{Graph of $F$ and numerical values of the minimal eigenvalue of $A$, $\lambda_{\min}(A)$ in function of $\Pa$, in the case $b=1$, for $P=50$ (blue circles) and $P=500$ (red crosses)}
	\label{figure:Fconditionnement}
\end{figure}
\begin{proof}
	Let $u \in \text{span}\{e_1,e_2,\cdots,e_P\}$ and let $\alpha$ its coordinates on this basis. Then 
	\[\alpha^T A \alpha = \int_{\mathcal{A}(a)} \!\!\!\!\! \nabla{u}^2 < \int_{B} \nabla{u}^2 = \norm{\alpha}_2^2,\]
	showing that the eigenvalues of $A$ are bounded by $1$. Thus $I-A$ is a positive symmetric matrix and the smallest eigenvalue of $A$ is bounded by
	\[\lambda_{\min}(A) \geq 1 - \text{tr}(I - A),\]
	which yields:
	\begin{eqnarray*}
		\lambda_{\min}(A) &\geq& 1 - 2\pi \sum_{p=1}^{P}C_p^2\int_{0}^{\rho_p a} u J_1(u)^2 du.
	\end{eqnarray*}
	We now use  \autoref{ConjCp}, which implies
	\[2\pi C_p^2 \leq \dfrac{1}{p} + \dfrac{1}{3p^2},\]
	combined with \eqref{EncadrementRhop}, to get
	\[\lambda_{\min}(A) \geq 1 - \int_{0}^{P} \frac{dt}{t} \int_{0}^{\pi t a} u J_1(u)^2du \hspace{1pt} - \frac{1}{3}\sum_{p=1}^P \frac{1}{p^2} \int_{0}^{\rho_p a} u J_1(u)^2du.\]
	For the first term, we write 
	\[\int_0^t u J_1(u)^2 = \frac{t^2}{2} \left(J_1^2(t) - J_0(t)J_2(t)\right),\] 
	while for the second, we use the classical inequality $J_1(u) \leq \frac{u}{2}$ to deduce
	\[\lambda_{\min}(A) \geq 1 - \int_{0}^{\pi\gamma} \dfrac{t}{2} \left(J_1(t)^2 - J_0(t)J_2(t)\right)dt - \frac{1}{48} \sum_{p=1}^P \frac{(\rho_p a)^4}{p^2}.\]
	We use again \eqref{EncadrementRhop} to obtain:
	\[\lambda_{\min}(A) \geq 1 - F(\gamma) - \gamma^4\frac{\pi^4}{48} \frac{1}{P^4} \sum_{p=1}^P p^2,\]
	which obviously implies the claimed result.
\end{proof}
\subsubsection*{Conditioning of $A$ for large $\Pa$}																
The behavior of $\lambda_{\min}(A)$ is more difficult to study for large $\Pa$. Nevertheless, we observed the following exponential decay: for any $P \geq 10$, and $\gamma \geq 1.4$, the minimal eigenvalue of $A$ is bounded below by
\begin{equation}
	\label{conjLambdaMin}
	\lambda_{\min}(A) \geq 180 \exp(-5.8\gamma).
\end{equation}
\begin{figure}[H]
	\centering
	\includegraphics[scale = 1]{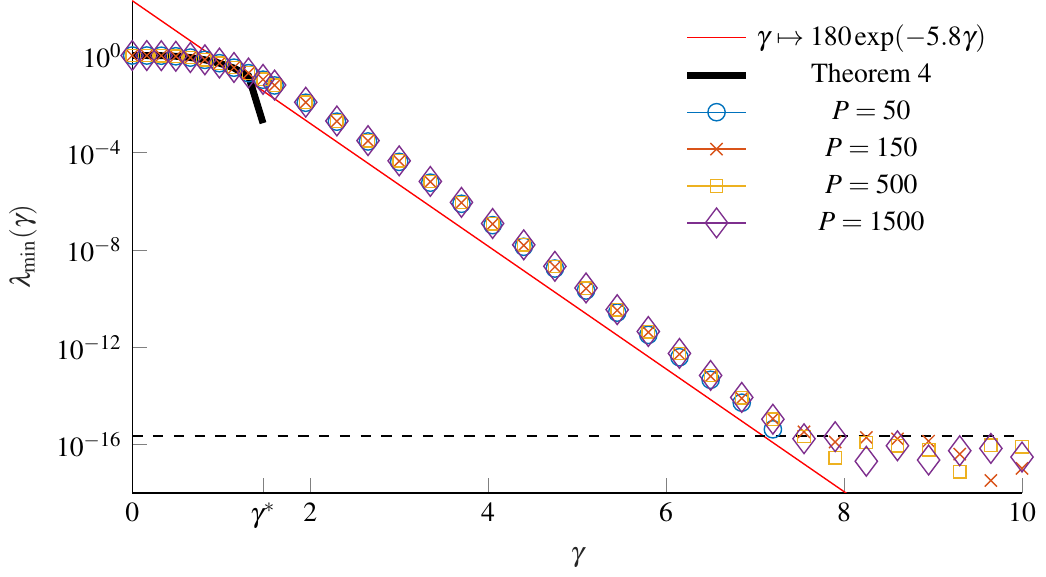}
	\caption{Lower bounds from \autoref{The:lowBoundCon} and \autoref{conjLambdaMin}, and estimated minimal eigenvalue of $A$ in function of $\gamma$ for $P=50$ (blue circles), $P=150$ (red crosses), $P=500$ (yellow squares) and $P=1500$ (purple diamonds). For $P=50$, the computed minimal eigenvalue is negative due to numerical errors from $\gamma \approx 7.5$ so corresponding data cannot be displayed. The horizontal dashed line shows machine precision}
\end{figure}

\section{Application to Laplace and Helmholtz kernels}
\label{sec:ApplicationLaplaceHelmholtz}
\subsection{Laplace kernel}
When solving PDE's involving the Laplace operator (for example in heat conduction or electrostatic problems), one is led to \eqref{discreteConv} with the Laplace kernel $G(r) = \log(r)$ (we have dropped the $-\frac{1}{2\pi}$ constant for simplicity). Here we show that its SBD converges exponentially fast:
\begin{Theorem} 
	\label{theRadialQuadLaplaceErreur}
	There exist two positive constants $L_1$ and $l_2$ such that
	\[ \forall a \in (0,1), \forall P \in \N^*, \forall r \in (a,1), \quad \abs{G(r) - \sum_{p=1}^P \alpha_p e_p(r)} \leq L_1 e^{-l_2 a P} \]
	where $\alpha_1,\cdots,\alpha_P$ are the SBD coefficients of $G$ of order $P$.  
\end{Theorem}
																		
We show this by exhibiting an extension $\tilde{G}$ for which we are able to estimate the remainder of the Fourier-Bessel series.
For any $n \in \N^{*}$, let us define extensions $\tilde{G}_n$ of $G$ as
\begin{equation}
	\tilde{G}_n = \begin{cases}
	r^{2n}\sum_{k=0}^{2n} \dfrac{a_{k,n}}{k!}(r-a)^k &\text{ if }r \leq a, \\
	G(r) &\text{ otherwise.}
	\end{cases}
\end{equation}
where the coefficients $a_{k,n}$ are chosen so that $\tilde{G}_n$ has continuous derivatives up to the order $2n$:
\[a_{k,n} = {\dfrac{d^k}{dr^k}\left(\dfrac{\log(r)}{r^{2n}}\right)\bigg|}_{r=a}.\]
Notice that the $r^{2n}$ term ensures the boundedness of $(-\Delta)^n \tilde{G}_n$ near the origin. Also observe that for all $s \in \N$:
\[(-\Delta)^s \tilde{G}_n \text{ vanishes on } \mathcal{C},\]
since $\tilde{G}_n \equiv G$ in a vicinity of $\mathcal{C}$. 
We now go into some rather tedious computations to provide a crude bound for $\norm{(-\Delta)^n \tilde{G}_n}_{L^2(B)}$ in terms of the coefficients $a_{k,n}$.
																		
\begin{lemma} 
	\label{LemmeDegueu}
	There exists a constant $C$ independent of $n$ and $a$ such that for $r<a$
	\begin{equation}
		\left|\Delta^n \tilde{G}_n(r)\right| \leq  C \left( \frac{16n}{e}\right)^{2n}\!\!\!\!\!\max_{k\in \left\{1,\cdots,2n\right\}}\left(\dfrac{|a_{k,n}|}{k!}a^k\right).
		\label{bigBadEq1Reduced}
	\end{equation}
	\label{LemAkDeltanf}
\end{lemma}
																		
\begin{proof} For $r \leq a$, we have
	\[\Delta^n \tilde{G}_n(r) = \sum_{k=0}^{2n}\sum_{l=0}^k \dbinom{k}{l}\dfrac{a_{k,n}}{k!}(-a)^{k-l}(2n+l)^2 (2(n-1)+l)^2\times ... \times (2+l)^2 r^{l}.\]
	This result is obtained by expanding the sum in the definition of $\tilde{G}_n$ and using the fact that $\Delta r^k = k^2r^{k-2}$. Hence, using triangular inequality
	\[|(-\Delta)^n \tilde{G}_n(r)| \leq \sum_{k=0}^{2n}\sum_{l=0}^k \dbinom{k}{l}\dfrac{|a_{k,n}|}{k!}a^{k-l}(2n+l)^2(2(n-1)+l)^2\times ... \times (2+l)^2r^{l}.\]	
	For $l\in \{1,\cdots,2n\}$, we apply the following (crude) inequality:
	\begin{equation}
		(2n+l)^2(2(n-1)+l)^2\times ... \times (2+l)^2 \leq (4n)^2(4n-2)^2 \times ... \times (2n+2)^2
		\label{estimationTresGrossiere}
	\end{equation}
	to obtain: 
	\begin{eqnarray*}		
		|(-\Delta)^n \tilde{G}_n(r)| &\leq& (4n)^2(4n-2)^2 \times ... \times (2n+2)^2\times\\
		&&\times\max_{k\in\llbracket 0,2n\rrbracket}\left(\dfrac{|a_{k,n}|}{k!}a^k\right)\sum_{k=0}^{2n}\sum_{l=0}^k \dbinom{k}{l}a^{-l}r^l	
	\end{eqnarray*}
	Since $r<a$, the last sum is bounded by $\displaystyle\sum_{k=0}^{2n}2^k = 2^{2n+1}-1 < 2^{2n+1}$,
	while 
	\[(4n)^2(4n-2)^2\times...\times (2n+2)^2 \sim 2\left(\dfrac{8n}{e}\right)^{2n}\]
	follows from Stirling formula. 
	
\end{proof}
We are now able to prove the following, which implies \autoref{theRadialQuadLaplaceErreur}.
\begin{Theorem}
	\label{The:DecroissanceErreurProlongementPoly}
	There exists a constant $C$ such that, for any $P \in \N^*$ and $a \in (0,1)$, there exists a radial function $\tilde{G}$ which coincides with $G$ on $\mathcal{A}(a)$ satisfying:
	\[\norm{\tilde{G} - \sum_{p=1}^{P}c_p(\tilde{G})e_p}_{\Hzrad} \hspace{-0.7cm}\leq C \sqrt{P} \exp\left(-\frac{aP\pi}{32}\right).\]
							
\end{Theorem}
\begin{proof}
	Let $n \in \N^*$, by Leibniz formula, 
	\begin{eqnarray*}						
		\dfrac{d^k }{dr^k}\left(r^{-2n}\log(r)\right) & = & \dfrac{(-1)^k k!}{r^{2n+k}}  \left(-\displaystyle\sum_{j=0}^{k-1}\dfrac{\dbinom{2n+j-1}{j}}{k-j}+\dbinom{2n+k-1}{k}\log(r)\right). \\
	\end{eqnarray*}
	This leads to \[\dfrac{|a_{k,n}|}{k!}a^k \leq a^{-2n} \dbinom{2n + k -1}{k}\left(\frac{k}{2n}-\log(a)\right),\]
	where we have used the identity
	\begin{equation*}
		\sum_{j=0}^{k-1}\dbinom{j+2n-1}{j} = \dfrac{k}{2n}\dbinom{k+2n-1}{k}.
	\end{equation*}
	Observe that
	\begin{equation*}
		\dbinom{2n+k-1}{k}\leq \dbinom{4n-1}{2n} = \frac{1}{2}\dbinom{4n}{2n} \leq \dfrac{4^{2n}}{2\sqrt{2\pi n}} \quad k \in \{1,\cdots,2n\},
	\end{equation*}
	and thus,
	\begin{equation}
		\max_{0\leq k \leq 2n}\left(\dfrac{|a_{k,n}|}{k!}a^k\right) \leq \left(\frac{4}{a}\right)^{2n}\dfrac{1}{2\sqrt{2\pi n}}\left(\log\left(\frac{e}{a}\right)\right).
		\label{majorAkLog} 
	\end{equation}							
	Combining (\ref{majorAkLog}) with estimation (\ref{bigBadEq1Reduced}), we find that there exists a constant $C$ such that, for $r<a$
	\[|(-\Delta)^n \tilde{G}_n (r)|\leq \dfrac{C}{\sqrt{n}}\left( \frac{16n}{e}\right)^{2n}\left(\frac{4}{a}\right)^{2n}\log\left(\dfrac{e}{a}\right).\]
	Therefore, integrating on $B(0,a)$, we get
	\[ \norm{(-\Delta)^n \tilde{G}_n}_{L^2(B(0,a))} \leq \dfrac{C a^2}{\sqrt{n}}\log\left(\frac{e}{a}\right)\left( \frac{64n}{ae}\right)^{2n}\!\!\!\!\!\!,\]
	and since 
	\[(-\Delta)^n \tilde{G}_n(x) = (-\Delta)^n G(x) = 0\]
	for $|x|>a$, the same bound applies to $\norm{(-\Delta)^n \tilde{G}_n(x)}_{L^2(B)}$. 
	We now plug this estimate into the inequality of corollary \ref{EstimationRest}, to get
	\[ \norm{\tilde{G}_n - \sum_{p=1}^{P}c_p(\tilde{G}_n)e_p}_{\Hzrad} \!\!\!\!\!\!\!\!\!\!\leq~~ C \dfrac{P^\frac{3}{2}}{n} a^2 \log\left(\dfrac{e}{a}\right)\left( \frac{64 n}{ae P \pi}\right)^{2n}\!\!\!\!\!\!.\] 
	The previous inequality holds true for any integer $n$ such that $n>1$ and any $P \in \mathbb{N}$. Without loss of generality, one can assume that $\frac{aP\pi}{64} >1$. In this case, let $n_P = \lfloor \frac{aP\pi}{64}\rfloor $, and $\tilde{G} = \tilde{G}_{n_P}$. Using the fact that $x\mapsto x \log\left(\dfrac{e}{x}\right)$ is bounded on $(0,1]$, we get 
	\[ \norm{\tilde{G} - \sum_{p=1}^{P}c_p(\tilde{G})e_p}_{\Hzrad} \leq C \sqrt{P} e^{-\frac{aP\pi}{32}}.\]
	
\end{proof}																												
\begin{Remark}
	The convergence rate is indeed bounded by a function of the parameter $\Pa$. \autoref{ApplicationNumLaplace} shows the decay of the $L^\infty$ error in function of $\Pa$ for different values of $P$. It can be seen that the error consistently decreases at an exponential rate of about $l_2 \approx 3.7$, and stagnates at the minimal error $e_{\min} \approx 10^{-10}$. We believe this stagnation is due to rounding errors related to the increasing condition number of the matrix $A$. However, note that the situation is invariant for constant $\gamma$, thus the stability of the linear system only depends on the target tolerance and not on the size of the problem. For example, to reach the error level $\varepsilon = 10^{-3}$, it is sufficient to take $\gamma \approx 1.8$ and the conditioning of $A$ is about $200$, independent on the specific value of $P$ and $a$.  
\end{Remark}
\begin{figure}[ht]
	\newlength{\plotwidth}
	\setlength{\plotwidth}{0.49\textwidth}
	\centering		
	\begin{subfigure}[b]{\plotwidth}
		\centering
		\includegraphics[scale = 0.85]{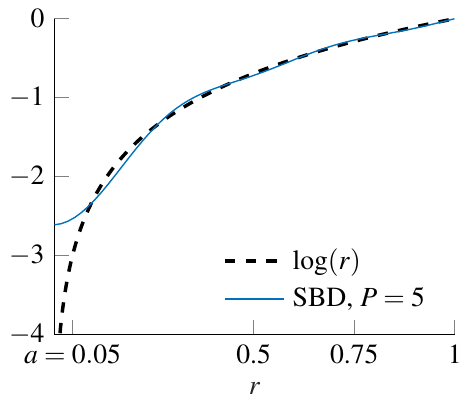}
		\caption{}
	\end{subfigure}%
	\hspace{1.5pt}
	\begin{subfigure}[b]{\plotwidth}
		\centering
		\includegraphics[scale = 0.85]{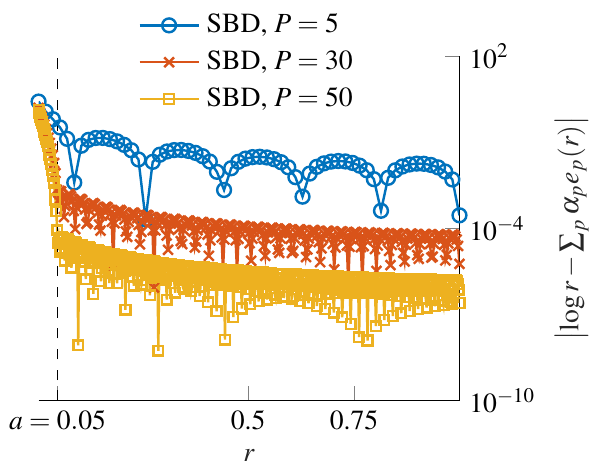}
		\caption{}
	\end{subfigure}%
																																																																								
	\begin{subfigure}[b]{\plotwidth}
		\centering
		\includegraphics[scale = 0.85]{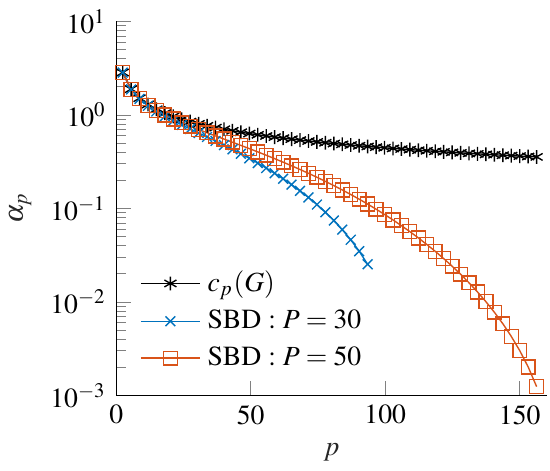}
		\caption{}	
	\end{subfigure}%
	\begin{subfigure}[b]{\plotwidth}
		\centering
		\includegraphics[scale = 0.85]{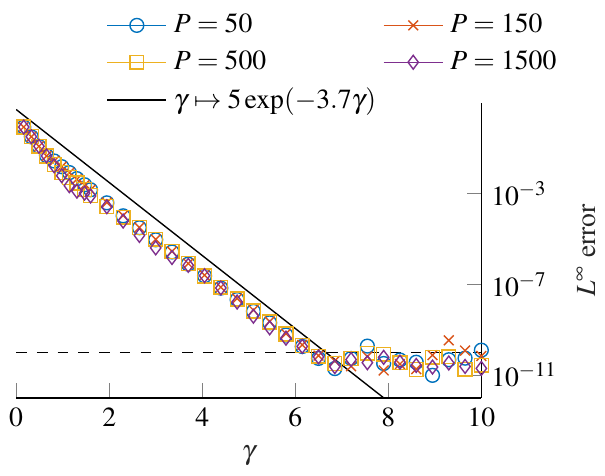}
		\caption{}				
		\label{ApplicationNumLaplace}
	\end{subfigure}%
	\caption{Approximation of $G(r) = \log r$ by SBD. \textbf{(a):}$L^{\infty}$ error between $\log(x)$ (dashed black) and its SBD approximation (solid blue) with $a=0.05$ and $P=5$. \textbf{(b):} Values of $\log r - \sum_{p}\alpha_pe_p(r)$ where $\alpha_p$ are the SBD coefficients for $a=0.05$ of order $P=5$ (Blue circles), $P=30$, (red crosses) and $P=50$ (yellow squares). \textbf{(c):} Numerical values of the SBD coefficients $\alpha_p$ in function of $p$, for $P=5$ (blue circles), $P=30$ (red crosses) and $P=  50$ (yellow squares). The black stars show the values of the exact (slowly decaying) Fourier-Bessel coefficients of $G$, that is $\alpha_p = c_p(G) = \frac{1}{\sqrt{\pi}\rho_p \abs{J_1(\rho_p)}}$. \textbf{(d):} Evolution of the $L^{\infty}$ error over $[a,1]$ associated to the $SBD$ for different values of $P$ in function of $\Pa$. An exponential decay is indeed observed, at the roughly estimated rate of $\propto \exp(-3.7\Pa)$. The stops decreasing at $e_{\min} \approx 10^{-10}$, for a value of $\Pa \approx 6.7$ because of the ill conditioning of the linear system \eqref{LinearSystem}}
	\label{MultiFig}
\end{figure}

\subsection{Helmholtz kernel}
\label{HelmoholtzSubSection}
Let $Y_0$ the classical Bessel function of second kind and of order $0$. For any $k>0$, the Helmholtz kernel, $r \mapsto \frac{-i}{4}H^{(1)}(kr)$, where $H^{(1)}(r) = J_0(r) + i Y_0(r)$, is the fundamental Green's kernel associated to the harmonic wave operator $- \Delta - k^2$, that satisfies a Sommerfeld radiation condition at infinity, (see for example  \cite{wilcox1975scattering}). This kernel arises in various physical problems, such as sound waves scattering. To approximate $H^{(1)}(kr)$ as a sum of dilated $J_0$ functions, it is sufficient to produce a SBD decomposition of $r \mapsto Y_0(kr)$. We can obtain good approximations of $Y_0$ in series of dilated functions in the following way:
\begin{itemize}
	\item[-]\textbf{When $k$ is a root of $Y_0$}: In this case the multi-Dirichlet condition is satisfied at any order. Indeed, for any $n$, 
	\[(-\Delta)^n Y_0(kr)\big|_{r=1} = k^{2n} Y_0(k) = 0.\]
	We thus produce a SBD decomposition of $Y_0$ on an interval $(a,1)$. Just like for the Laplace kernel, it was observed that the approximation error converges exponentially fast to zero, as soon as $P$ is greater than $k$. 
	\item[-]\textbf{When $k$ is close to, or greater than the first root of $Y_0$}: we find a Sparse Bessel Decomposition for $r \mapsto Y_0(k'r)$ on $(a,1)$, where $k'$ is the first root of $Y_0$ larger than $k$. This provides a decomposition for $r \mapsto Y_0(kr)$ valid on $(\frac{k'}{k}a,\frac{k'}{k})$.
	\item[-]\textbf{When $k$ is much smaller than the first root of $Y_0$}: the previous idea might lead to unnecessary efforts. Indeed, to ensure that $\frac{k'}{k}a$ is small enough, one would have to choose a very small value of $a$ leading to a very long Bessel series. Instead, one can use the Bessel-Fourier series associated to the Robin condition (see \autoref{Robin}):
	\[\dfrac{\partial u}{\partial n} + H u = 0,\]
	noticing that $H = -\dfrac{k Y_0'(k)}{Y_0(k)} > 0$ in this region. 
\end{itemize}
\noindent

\begin{figure}[ht]	      	
	\setlength{\plotwidth}{0.7\textwidth}
	\centering	
	\includegraphics[scale = 1]{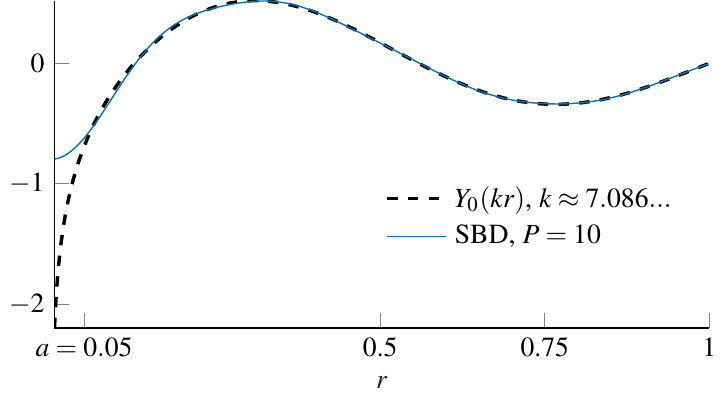}
	\caption{SBD of $Y_0(kr)$ with $k$ a root of $Y_0$ approximately equal to $7.086$, with $P=10$ terms and $a=0.05$ }	      
	\label{Y0Example}
\end{figure}

\subsection{General kernel : enforcing the multi-Dirichlet condition}
\label{begal1}
For general kernels $G$, the multi-Dirichlet conditions may not be fulfilled, even after rescaling. When applying the SBD method without any changes, this leads to the situation observed in the top panel of \autoref{fig:badSit}. In this figure, the SBD is applied to the kernel $G_1(r) = \log(r) + \sin(r)$ (note that $\Delta G_1(1) \neq 0$), and we plot the error in function of $r$ for several values of $\gamma$ and $P = 10$. The error curve stagnates near $r=1$. In the bottom panel, we apply the SBD to $G_2(r) = \log(r) + \sin(r) - \frac{1}{4}\Delta G_1(1)r^2$ where the last term enforces the Dirichlet condition, and the error decay is improved. 

\begin{figure}[ht]	
	\centering
	\setlength{\plotwidth}{\textwidth}
	\begin{subfigure}[b]{\plotwidth}
		\centering
		\includegraphics[scale = 1]{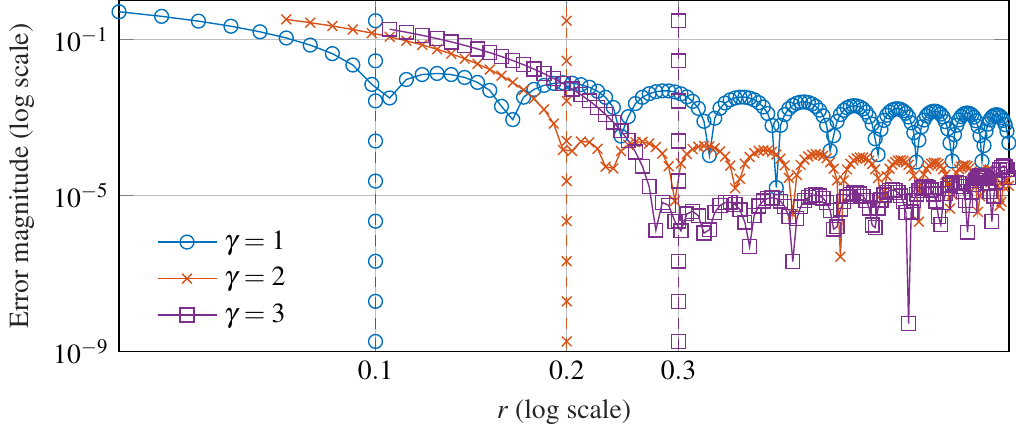}
		\caption{$G_1(r) = \log(r) + \sin(r)$}
	\end{subfigure}
	\vspace{5.0pt}

	\begin{subfigure}[b]{\plotwidth}
		\centering
		\includegraphics[scale = 1]{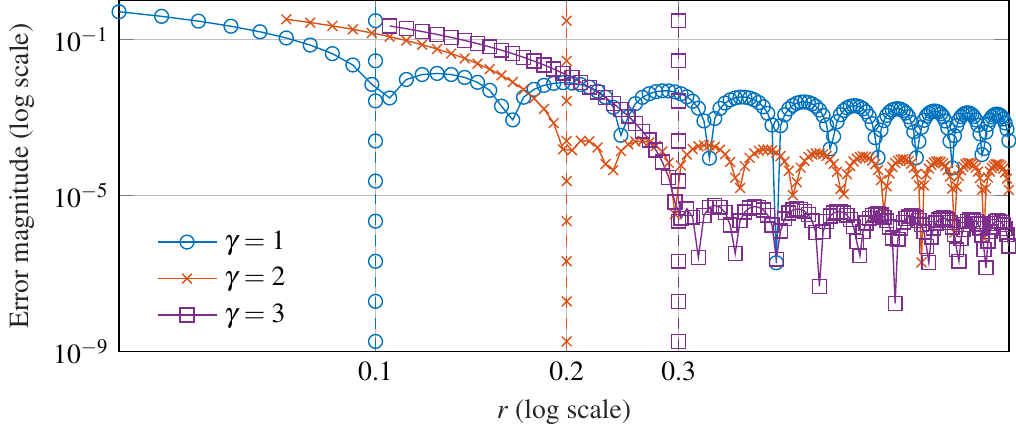}
		\caption{$G_2(r) = \log(r) + \sin(r) - \frac{1}{4}\Delta G_1(1)r^2$}		
	\end{subfigure}
	\caption{SBD of order $10$ applied to the kernels $G_1(r) = \log(r) + \sin(r)$ (top panel) and $G_2(r) = \log(r) + \sin(kr) - \frac{1}{4}\Delta G_1(1)r^2$ (bottom panel), for $\gamma = 1$ (blue circles), $\gamma = 2$ (red crosses) and $\gamma = 3$ (purple squares). The vertical dashed lines show the position of $a$ for each $\gamma$.}
	\label{fig:badSit}
\end{figure}

More generally, for any radial $G$, we can apply the SBD to a modified function $H = G - K$ where $K$ is chosen to enforce the multi-Dirichlet condition. Since we wish to obtain for $G$ a decomposition in sum of Bessel functions, we propose to choose $K$ in the form 
\[K(r) = \sum_{t=1}^{n} \mu_t J_0(\omega_t r)\]
for some $\varInRange{\omega}{t}{1}{n}$ that are not roots of $J_0$. It is not the aim of this paragraph to describe a systematic way of choosing $\varInRange{\omega}{t}{1}{n}$, as this work is still in progress. However, when the first few iterates of the Laplace operator on $G$ are known, we suggest the following choice: let $\omega'$ the square root of the average ratio between too successive iterates, choose $\omega_1$ as the root of $J_1$ that is closest from $\omega'$. Then, assign $\omega_2, \cdots, \omega_n$, successively to the closest roots of $J_1$. The coefficients $(\mu_t)_{1 \leq t \leq n}$ are finally found by inverting a small linear system 
\[M\mu = \lambda,\]
where $\lambda$ is the vector given by
\begin{eqnarray*}
	\lambda_t &=& (-\Delta)^t G \big|_{r=1}, \quad t\in \{1,\cdots,n\},
\end{eqnarray*}
and with
\[M=
	\begin{bmatrix}
		-\omega_1^2      & -\omega_2^2             & \cdots & -\omega_n^2  \\
			\omega_1^4   & \omega_2^4 		              & \cdots & \omega_n^4   \\
		\vdots      	& \vdots                              & \cdots & \vdots     \\
		(-1)^n \omega_1^{2n}      & (-1)^n \omega_2^{2n}  & \cdots &(-1)^n \omega_n^{2n} \\ 
	\end{bmatrix}
\]
\vspace{5pt}

In \autoref{figArbitraryKernel}, we show the efficiency of this method by applying the SBD with $100$ terms to some highly oscillating function ($x \mapsto \log(x) + \sin(250x)$), 
for $n \leq 3$, and computing the maximal error of the decomposition in function of $\gamma$. The frequencies $\varInRange{\omega}{t}{1}{n}$ are the roots of $J_1$ that are closest to $250$. 
										
\begin{figure}[H]	
	\setlength{\plotwidth}{0.7\textwidth}
	\centering
	\includegraphics[scale = 0.8]{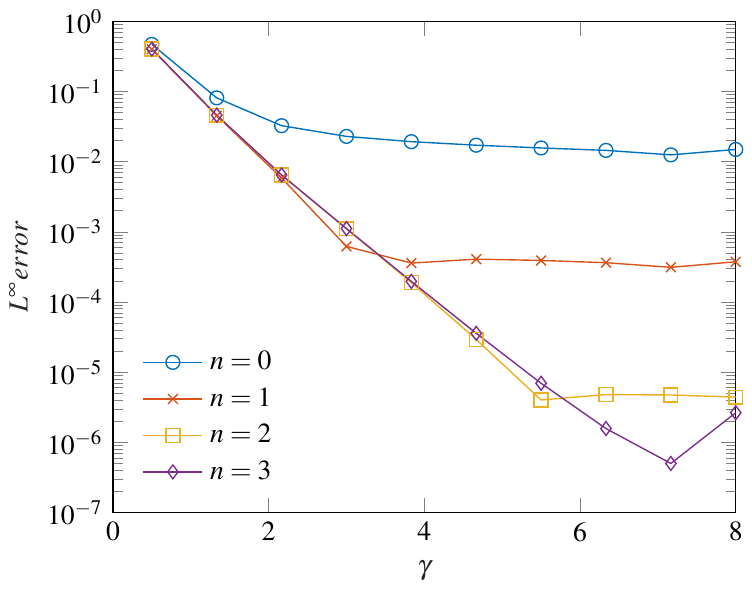}
	\caption{Maximal error betwwen the kernel $G_1(r) =  \log(x) + \sin(250x)$ and its $100$-terms SBD in function of $\gamma$ using the method described in this paragraph for several values of $n$.}
	\label{figArbitraryKernel}
\end{figure}																									
																														
\section{Circular quadrature}
\label{sec:circular}
In this section, we study the approximation of the form
\[ J_0(\rho_p|x|) \approx \dfrac{1}{M_p}\sum_{m=0}^{M_p-1}e^{i \rho_p \bs{\xi}^p_m \cdot x}, \]
for some integer $M_p$ and some quadrature points $(\bs{\xi}_m^p)_{1 \leq m \leq M_p}$. 
																														
\subsection{Theoretical bound}
\begin{Theorem} There exists a constant $K$ such that for any $r>0$, $M\in \N^*$, such that $M \geq \frac{e}{2}r$, and for any $\varphi \in \R$ 
	\[\left|J_0(r) -  \dfrac{1}{M}\sum_{m=0}^{M-1}e^{ir\sin\left(\frac{2m\pi}{M}-\varphi\right)} \right| \leq K \left(\dfrac{er}{2M}\right)^M.\]
	\label{QuadratureCirc}
\end{Theorem}
\noindent In order to prove this proposition, we first prove a result on Fourier series
\begin{lemma} For any $2\pi-$periodic $\mathcal{C}^2$ complex-valued function $f$ one has 
	\[\dfrac{1}{2\pi}\int_{0}^{2\pi}f - \dfrac{1}{M}\sum_{m=0}^{M-1}f\left(\frac{2m\pi}{M} \right) = - \sum\limits_{k \in \Z^*}c_{kM}(f),\]
	where $c_n(f)$ denotes the Fourier coefficient of $f$ defined as 
	\[c_n(f) = \dfrac{1}{2\pi}\int_{0}^{2\pi}f(x)e^{-inx}dx.\]
\end{lemma}
\begin{proof}
	Since $f$ is $\mathcal{C}^2$, it is equal to its Fourier Series, which converges normally: \[\forall x \in \mathbb{R}, f(x) = \sum_{k\in\Z} c_k(f)e^{ikx}.\] Using this expression, we obtain \[\dfrac{1}{M}\sum_{m=0}^{M-1}f\left(\frac{2m\pi}{M}\right) = \sum\limits_{k\in \Z^*}c_k(f)\left(\frac{1}{M}\sum_{m=0}^{M-1}e^{ik\frac{2m\pi}{M}}\right).\] 
	Now observe that
	\[\dfrac{1}{M}\sum_{m=0}^{M-1}e^{ik\frac{2m\pi}{M}} =   \begin{cases}
	1 &\text{ if } k \notin M\Z, \\
	0 &\text{ otherwise.}
	\end{cases}\] 
	Therefore \[\int_{0}^{2\pi}f(x)dx - \dfrac{1}{M}\sum_{m=0}^{M-1}f\left(\frac{2m\pi}{M} \right) = c_0(f) - \sum\limits_{k \in M\Z}c_{k}(f) = - \sum\limits_{k \in \Z^*}c_{kM}(f).\]
	
\end{proof}
\noindent We now turn to the proof of \autoref{QuadratureCirc}: 
\begin{proof}
	The result is based on the following fact: 
	\[J_0(r) =  \int_0^{2\pi} e^{ir\sin(x)}dx = \int_0^{2\pi} e^{ir\sin(x - \varphi)}dx.\] 
	Let $f : x \mapsto e^{ir\sin(x - \varphi)}$. Let us recall the integral representation of the Bessel function of the first kind and of order $k$ where $k$ is a relative integer: \[J_k(r) =  \int_{0}^{2\pi}e^{ir\sin(x)}e^{-ikx}dx =  e^{-ik\varphi}\int_{0}^{2\pi}e^{ir\sin(x - \varphi)}e^{-ikx}dx.\] Thus, one has $c_k(f) = e^{ik\varphi}J_k(r)$. Consequently, the former Lemma yields 
	\[J_0(r) -  \dfrac{1}{M} \sum_{j=0}^{M-1} e^{ir \sin \left( \frac{2j\pi}{M}-\varphi \right)} = -\sum_{k\in \Z^*}e^{iNk\varphi}J_{Nk}(r).\] 
	For large $\abs{k}$, 
	\[J_k(r) \sim \left(\dfrac{er}{2\abs{k}}\right)^{\abs{k}}.\]
	Therefore, there exists a constant $C'$ such that: 
	\begin{eqnarray*}
		\abs{J_0(r) -  \dfrac{1}{M}\sum_{m=0}^{M-1}e^{ir\sin\left(\frac{2m\pi}{M}-\varphi\right)}} &\leq& C' \sum_{k\in \Z^*} \left(\dfrac{er}{2M|k|}\right)^{M|k|}\\
		&\leq& K \left(\dfrac{er}{2M}\right)^M 
	\end{eqnarray*}
	for an appropriate choice of $K$.	
	
\end{proof}
We conclude with the following result
\begin{proposition} Let $\varepsilon >0$, $r>0$, and assume $M > \dfrac{e}{2}r + \log\left(\dfrac{K}{\varepsilon}\right)$. Then 
	\[\left|J_0(r) -  \dfrac{1}{M}\sum_{m=0}^{M-1}e^{ir\sin\left(\frac{2m\pi}{M}-\varphi\right)} \right| \leq \varepsilon.\]
	\label{suboptCirc}
\end{proposition}
\begin{proof}
	This result is a direct consequence of the previous proposition together with the following inequality: for any $(A,B) \in \left(\mathbb{R}_+^*\right)^2$ one has
	\[ \left( \dfrac{A}{A+B}\right)^{A+B} \leq e^{-B}.\]
	To prove it, we take the logarithm of this quantity, 
	\[f(A,B) = -B\left(1+\dfrac{A}{B}\right)\log\left(1+\dfrac{B}{A}\right)\]
	and observe that for any positive $x$, \[\left(1+\dfrac{1}{x}\right)\log(1+x) \geq 1.\]
	
\end{proof}
																																																		
Hence, we can approximate very efficiently the functions $e_p$ of the previous paragraph as a finite sum as follows. We define the quadrature points $\bs{\xi}_0^p, \bs{\xi}_1^p, ..., \bs{\xi}^p_{M_p-1}$  by
\begin{equation}
	\label{defXimp}
	\bs{\xi}_m^p := \displaystyle e^{i\frac{2\pi m}{M_p}} \quad 1\leq p \leq P, \quad 0 \leq m \leq M_p -1.
\end{equation}
With this definition, for any $\bs{x} \in \mathbb{R}^2$
\[ e_p(|\bs{x}|) = C_p J_0(\rho_p |x|)\approx \dfrac{C_p}{M_p}\sum_{m=0}^{M_p-1}{e^{i \rho_p \bs{x} \cdot \bs{\xi}_m^p}},\]
and the approximation is valid at a precision $\varepsilon$ as soon as $M_p > \frac{e}{2}\rho_p|\bs{x}| + \log\left(\dfrac{KC_p}{\varepsilon}\right)$.

\section{Estimations of complexities}
\label{sec:complexities}
We now turn to the complexity estimate of the complete algorithm given in \autoref{sec:overview}. We fix $\alpha \in [0,1/6]$ and let 
\begin{equation}
	\label{def_a}
	a = \dfrac{\abs{\log\varepsilon}^{2/3}}{N_z^{2/3 - \alpha}}
\end{equation}  
as in \autoref{The:GlobalComplexity}. We will give a bound for the number of operations of each step of the algorithm in function of $N_z$, $\varepsilon$ and $\alpha$. We note  $C_{\textup{SBD}}$, $C_{\textup{circ}}$, $C_{\textup{assemble}}$, $C_{\textup{far}}$ and $C_{\textup{close}}$ respectively, the number of operations required  to produce the SBD, the circular quadrature, to assemble the close correction matrix $D$ defined in \eqref{defD}, to compute the far approximation defined in \eqref{FarApprox}, and to apply $D$ on a vector. We will denote by $C$ any positive constant that is independent of $N_z$, $\varepsilon$ and $\alpha$. 
\subsection{Offline computations}
The first part of the algorithm consists in combining the SBD with the circular quadratures detailed in the previous two sections to derive an approximation scheme for the $\log$ function in the following form: 
\[ \log |\bs{x}| = \sum_{\nu=1}^{N_\xi} \hat{w}_{\nu} e^{i \bs{x}\cdot \bs{\xi}_{\nu}} \quad a < \abs{\bs{x}} < 1\]
valid to the accuracy $\varepsilon$.
																																																		
\paragraph{Sparse Bessel Decomposition:}
We first compute a SBD of $\log$ on the ring $\{a < r < 1\}$ to reach the accuracy $\frac{\varepsilon}{2}$, as  developed in Sections \ref{sec:SBD} and \ref{sec:ApplicationLaplaceHelmholtz}. We write this approximation
\[ \log r \approx \sum_{p=1}^P \alpha_p e_p(r), \quad a < r < 1.\]
\autoref{theRadialQuadLaplaceErreur} shows that the accuracy is reached for 
\begin{equation}
	\label{eq:valeurDePenFonctionDe_a}
	P = O\left(\dfrac{|\log(\varepsilon)|}{a}\right).
\end{equation}
Since the coefficients $\alpha_1,\cdots,\alpha_P$ are obtained through the inversion of a $P \times P$ matrix, the computation of the SBD requires $O(P^3)$ computations. Therefore, there exists a constant $C>0$ independent of $N_z$, $\varepsilon$ and $\alpha$ such that 
\begin{equation}
	\label{Complex:SBD}
	C_{\textup{SBD}}(N_z,\varepsilon,\alpha) \leq C \abs{\log\varepsilon} N_z^{2 - 3\alpha}.
\end{equation}
																																																		
\paragraph{Circular quadrature:} We approximate each function $e_p$ using a circular quadrature as detailed in \autoref{sec:circular}. For each $p$, we choose the number $M_p$ of terms in the quadrature so that
\begin{equation}
	\label{temp1}
	\abs{J_0 (\rho_p \abs{\bs{x}}) - \dfrac{1}{M_p}\sum_{m=1}^{M_p} e^{i \rho_p \bs{x} \cdot \bs{\xi}_m^p}} \leq \frac{\varepsilon}{2P \abs{\alpha_p}C_p} \quad a < \abs{\bs{x}} < 1,
\end{equation}
where the quadrature points $\bs{\xi}_m^p$ are defined in  \eqref{defXimp}. \autoref{suboptCirc} implies that taking
\begin{equation}
	M_p > \frac{e}{2} \rho_p + \log\left(\dfrac{2KP |\alpha_p|}{\varepsilon}\right)
\end{equation} 
is sufficient to ensure \eqref{temp1}. In this case, triangular inequality implies that for $a \leq \abs{\bs{x}} \leq 1$, 
\[ \abs{\log\abs{\bs{x}} - \sum_{p=1}^{P} \sum_{m = 1}^{M_p} C_p \frac{\alpha_p}{M_p}e^{i \bs{x}\cdot \bs{\xi}^p_m}} \leq \varepsilon.\]
Bessel's inequality
\[\sum_{p=1}^P{|\alpha_p|^2} \leq \int_{B} \abs{\log|\bs{x}|}^2\]
ensures the boundedness of $\alpha_p$. Moreover, $\rho_p = O(p)$, implying $M_p = O(P)$, and hence, 
\begin{equation}
	\label{eq:NxiEnFonctionDeP}
	N_\xi = \sum_{p = 1}^P M_p = O(P^2).
\end{equation}
Since for any $p$, the computation of $(\bs{\xi}_m^p)_{1\leq m \leq M_p}$ has a linear complexity in $M_p$, we get:
\begin{equation}
	\label{complex:circ}
	C_{\textup{circ}}(N_z,\varepsilon,\alpha) \leq C \abs{\log\varepsilon}^{2/3} N_z^{4/3 - 2\alpha}.
\end{equation} 
Equations \eqref{Complex:SBD} and \eqref{complex:circ} yield the first part of \autoref{The:GlobalComplexity}.
																																																		
\paragraph{Close correction matrix:} Recall that $\rmax$ is defined as 
\[\rmax = \max_{1\leq k,l\leq N_z} |\bs{z}_k - \bs{z}_l|\]
where $\bs{z}$ are the data points in \eqref{discreteConv}. Let $\rmin = a \rmax$. We first determine the set $\mathcal{P}$ of all pairs $(k,l)$ such that $\abs{\bs{z}_k - \bs{z}_l} \leq \delta_{\min}$. This is the classical "fixed-radius near neighbors search", and can be solved in $O(N_z \log N_z + \# P)$ operations (see for example \cite{bentley1975multidimensional, bentley1977complexity,turau1991fixed,dickerson1990fixed}). In order to compute an approximation of the close correction sparse matrix:
\[D_{kl} = \delta_{(k,l) \in \mathcal{P}} \left( \log\abs{\bs{z}_k - \bs{z}_l} - \sum_{\nu = 1}^{N_{\xi}}e^{i (\bs{z}_k - \bs{z}_l)\cdot \boldsymbol{\xi}_\nu} \hat{\omega}_\nu\right),\]
we need $\#\mathcal{P}$ evaluations of $\log$, and the computation of 
\[\operatorname{NUFFT}_-[\left(\bs{z}_k - \bs{z}_l\right)_{(k,l)\in\mathcal{P}},\bs{\xi}]\big(\hat{\nu})\]
at precision $\varepsilon$. For data uniformly distributed on a curve, the number of close pairs scales as 
\begin{equation}
	\label{eq:NombreDinteractionsProches}
	\# \mathcal{P} = O\left(\dfrac{\rmin}{\rmax} N_z\right) = O(N_z^2 a).
\end{equation}
One can check that $N_\xi \leq N_z^2a$ using \eqref{def_a}, \eqref{eq:valeurDePenFonctionDe_a} and \eqref{eq:NxiEnFonctionDeP}, so that 
\begin{equation}
	\label{Complex:assemble}
	C_{\textup{close}} \leq C C_{\textup{NUFFT}}(\varepsilon)\abs{\log{\varepsilon}}^{2/3}N_z^{4/3 + \alpha}\log(N_z).
\end{equation}
This is the second part of \autoref{The:GlobalComplexity}. 
\subsection{On-line Computations}
\paragraph{Far approximation:} Recall that for all $k \in \{1,\cdots,N_z\}$, the far approximation is defined by the following equation:
\[ q^{\text{far}}_k = \sum_{l=1}^{N_z} G_{\textup{approx}}(\bs{z}_k - \bs{z}_l) f_l,\]
where, according to the previous subsection,
\[G_{\textup{approx}}(\bs{x}) = \sum_{p=1}^P\sum_{m = 1}^{M_p} \dfrac{\alpha_pC_p}{M_p}e^{i \bs{x}\cdot \bs{\xi}^p_m}.\]
Define $\hat{\omega} = \varInRange{\hat{\omega}}{\nu}{1}{N_\xi}$ and $\bs{\xi} = \varInRange{\bs{\xi}}{\nu}{1}{N_\xi}$ such that
\[ G_{\textup{approx}}(\bs{x}) = \sum_{\nu=1}^{N_\xi} \hat{\omega}_\nu e^{i \bs{x} \cdot \bs{\xi}_\nu}.\]
To compute $q^{\textup{far}}$, recall the following three steps:
\begin{itemize}
	\item[(i)] \textbf{Space $\rightarrow$ Fourier: } Compute $\hat{f} = \textup{NUFFT}_-[\bs{z},\boldsymbol{\xi}](f),$
	\item[(ii)] \textbf{Fourier multiply} Perform elementwise multiplication by $\hat{\omega}$: 
	\[\hat{g}_{\nu} = \hat{\omega}_\nu \hat{f_\nu},\]
	\item[(iii)] \textbf{Fourier $\rightarrow$ Space: } Compute $q^{\text{far}} =  \textup{NUFFT}_+[\bs{z},\boldsymbol{\xi}](\hat{g}).$
\end{itemize}
One can check that \eqref{def_a} and \eqref{eq:NxiEnFonctionDeP} imply $N_{\xi} \geq N_z$, thus
\begin{equation}
	\label{Complex:far}
	C_{\textup{far}}(N_z,\varepsilon,\alpha) \leq C 	\abs{\log \varepsilon}^{2/3} C_{\textup{NUFFT}}(\varepsilon) N_z^{4/3 - 2\alpha } \log(N_z).
\end{equation}
\paragraph{Close correction:} $D$ has $\# \mathcal{P}$ non-zero entries so 
\begin{equation}
	\label{Complex:close}
	C_{\textup{close}}(N_z,\varepsilon,\alpha) \leq C \abs{\log\varepsilon}^{2/3}N_z^{4/3 + \alpha}.
\end{equation}
Summing \eqref{Complex:far} and \eqref{Complex:close} yields the second part of \autoref{The:GlobalComplexity}. 
																																																		
\begin{Remark}
	The extreme cases $\alpha= 0$ and $\alpha = 1/6$ correspond respectively to the situations where one wish to either minimize the total (off-line $+$ on-line) computation time or just the on-line time. The complexities "off-line" and "on-line" then become (omitting the dependence in $\varepsilon$):	
	\begin{table}[H]
		\centering
		\begin{tabular}{ |c|c|c| } 
			\hline
			               & Off-line                          & On-line        \\ 
			\hline
			$\alpha = 0$   & $O(N_z^2)$                        & $O(N_z^{4/3})$ \\ 
			$\alpha = 1/6$ & $O\left(N_z^{3/2}\log N_z\right)$ & $O(N_z^{3/2})$ \\ 
			\hline
		\end{tabular}
		\caption{Complexity of the algorithm (omitting dependence in $\varepsilon$) in the two extreme cases $\alpha=0$ and $\alpha = 1/6$.}
	\end{table}									
\end{Remark}
																																																		
\section{Numerical examples}
																																																		
\subsection{System of $N\times N$ particles}
																												
To assess for the numerical performance of our method, we first generate two sets of $N$ points $(\bs{x}_k)$ and $(\bs{y}_k)$ uniformly distributed in a square, for $N$ ranging from $10$ to $10^7$, and compute the discrete convolution  
\[ q_k = \sum_{l=1}^N \log\abs{x_k - y_l} f_l\]
where $f_l$ is a random vector in $\R^{N}$.
We measure both the time needed for off-line and on-line parts. We also measure the amount of memory occupied by the assembled operator (Memory usage goes a little bit above this value during the computation). The results are displayed in \autoref{tablePerf}. The computer used for this test is a laptop cadenced to 1.6 GHz and possessing 32 GB of memory. 
																																																		
																																																		
\begin{table}[H]
	\centering
	\begin{tabular}{lllll}
		\centering
		           & \vline\quad Off-line (s) & On-line(s) & Memory  & Proportion of full matrix size \\
		\hline					 				
		           & \vline                   &            &         &                                \\ 
		$N = 10$   & \vline\quad $0.034$      & $0.001$    & 7 kb    & $800\%$ (no compression)       \\
		$N = 10^2$ & \vline\quad $0.027$      & $0.002$    & 83 kb   & $105\%$ (no compression)       \\
		$N = 10^3$ & \vline\quad $0.049$      & $0.003$    & 1.02 Mb & $14\%$                         \\
		$N = 10^4$ & \vline\quad $0.37$       & $0.02$     & 10.7 Mb & $1.4\%$                        \\
		$N = 10^5$ & \vline\quad $3.4$        & $0.16$     & 109 Mb  & $0.15\%$                       \\
		$N = 10^6$ & \vline\quad $47.2$       & $2.6$      & 1.07 Gb & $0.015\%$                      \\
		$N = 10^7$ & \vline\quad $441$        & $41$       & 7.13 Gb & $9.8.10^{-4}\%$                \\		
	\end{tabular}
	\caption{Performances of the algorithm}
	\label{tablePerf}
\end{table}	
																																																		
\subsection{Sound canceling}
																																																		
Consider $N_z$ punctual 2-dimensional sound sources located at $\varInRange{\bs{z}}{l}{1}{N_z}$  and emitting at a single frequency $f$, with unit amplitude and phases $\varInRange{\varphi}{l}{1}{N_z}$. That is, for each $l \in \{1,\cdots, N_z\}$, the source number $l$ generated an acoustic pressure at each point $\bs{x}$ and time $t$ in $\R^2$ equal to 
\[p_l(\bs{x},t) = \Re\left[-\frac{i}{4}H^{(1)}(k\abs{\bs{z}_l - \bs{x}}) e^{-i\left(\omega t - \varphi_l \right)}\right],\]
where $\omega = 2\pi f$, $k = \frac{2\pi f}{c}$, with $c$ the celerity of the sound waves, and $H^{(1)}$ is the Hankel function of first kind already defined in \autoref{HelmoholtzSubSection}. 
By superposition, the resulting pressure at $\bs{x}$ is 
\[p(\bs{x},t) = \Re\left[ -\frac{i}{4}e^{-i \omega t} \sum_{l=1}^{N_z} H^{(1)}(k\abs{\bs{z}_l - \bs{x}})e^{i \varphi_l} \right],\]
and the sound intensity is proportional to
\[\Pi(\bs{x},\varphi) =  \abs{ \sum_{l=1}^{N_z} H^{(1)}(k\abs{\bs{z}_l - \bs{x}})e^{i \varphi_l} }^2.\]
Suppose one wishes to choose the phases that minimize the sound intensity in a prescribed zone $\Omega$ (called the silence zone), that is,
\[ \varphi^* = \underset{\varphi \in [0,2\pi]^{N_z}}{\text{argmin}}\int_{\Omega}\Pi(\bs{x},\varphi). \]
If we approximate the integral over $\Omega$ by a uniform quadrature, this leads to solving 
\[\varphi^* = \underset{\varphi \in [0,2\pi]^{N_z}}{\text{argmin}} \sum_{q=1}^{Q}\Pi(\bs{x}_q,\varphi), \]
where $\bs{x}_q$ are the coordinated of the quadrature points. If we let $A_{l,q} = H^{(1)}(k\abs{\bs{x}_q - \bs{z}_l})$ and $q(\varphi) = \left(e^{i\varphi_l}\right)_{1 \leq l \leq N_z}$, this rewrites
\[ \varphi* = \underset{\varphi \in [0,2\pi]^{N_z}}{\text{argmin}} q(\varphi)^T A^T A ~ ~ q(\varphi).\]
Using our method, the cost function associated to this minimization problem can be evaluated rapidly, as well as its gradient. Thus, using black-box optimization procedures, we can find rapidly good candidates for $\varphi^*$. In \autoref{figMinimizationHelmholtz}, we show the result of one such optimization, with $N_z=100$ sound sources randomly located on a half circle and where the zone of silence is represented by the red circle. The silence zone is discretized using a mesh of $Q = 2.5\times 10^5$ points. We stopped the optimization after $500$ evaluations of the objective function and its gradient, which required a total computation time of 3 minutes on our computer. Note that since the silence zone is located far away from the sources, we don't need to use a close correction matrix to compute the objective function. To produce the image in \autoref{figMinimizationHelmholtz}, the resulting field is evaluated at $2 \times 10^7$ grid points.
																																																		
\begin{figure}
	\centering	
	\includegraphics[scale = 0.45]{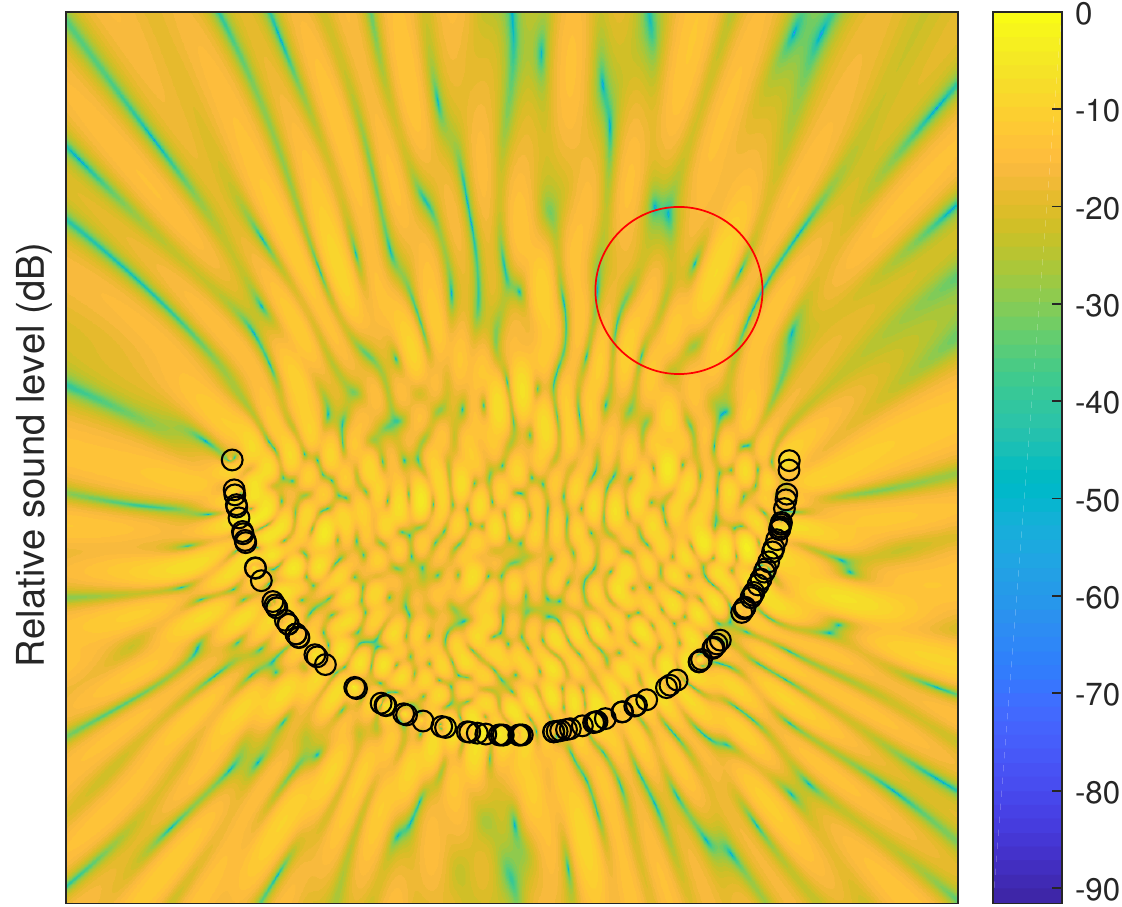}		
	\includegraphics[scale = 0.45]{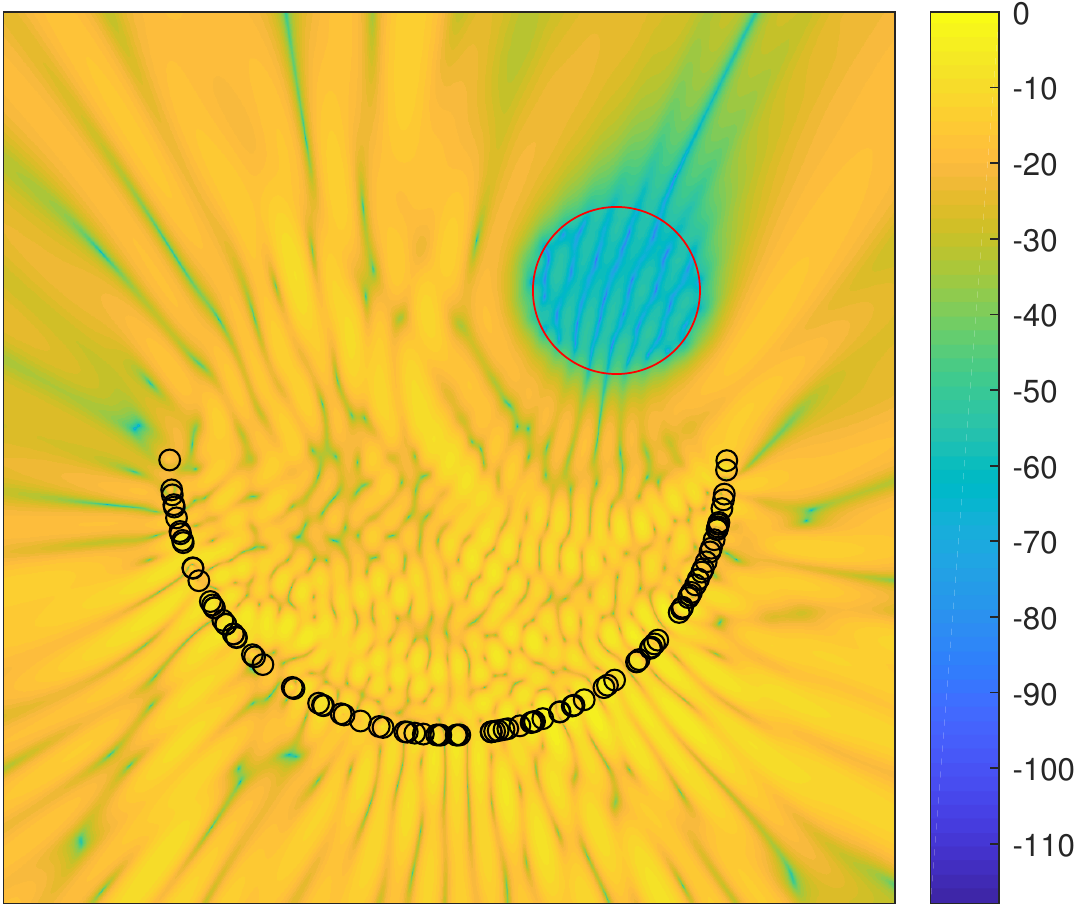}
	\caption{Sound level (in dB) created by $100$ punctual 2D sound sources with wavenumber $k\approx 90$ located at random spots on a half circle with radius $r = 0.5$. The sources locations are indicated by the small black circles. On the left, we show the sound level before the optimization of the phases. On the right, we show the result once the phases of the sources have been optimized to minimize the sound level in the region delimited by the red circle. }
	\label{figMinimizationHelmholtz}
\end{figure}

\bibliography{biblio}
\bibliographystyle{plain}

\end{document}